\documentclass[10pt]{amsart}
\usepackage{geometry}
\usepackage{amssymb}
\usepackage{stmaryrd}

\usepackage[ocgcolorlinks, linkcolor=blue]{hyperref}

\usepackage[ocgcolorlinks,linkcolor=blue]{hyperref}

\usepackage{ifdraft}
\ifdraft{
\newgeometry{asymmetric}
\usepackage{showkeys}
\usepackage[bordercolor=red, color=red, colorinlistoftodos]{todonotes}
}{
\usepackage[disable]{todonotes}
}

\makeatletter\providecommand\@dotsep{5}\def\listtodoname{List of Todos}\def\listoftodos{\hypersetup{linkcolor=black}\@starttoc{tdo}\listtodoname\hypersetup{linkcolor=blue}}\makeatother
\newtheorem{thm}{Theorem}
\newtheorem{lem}{Lemma}
\newtheorem{prop}{Proposition}

\newtheorem{assumption}{Assumption}
\newtheorem{rmk}{Remark}

\numberwithin{equation}{section}
\newcommand{\bel}{\begin{equation} \label}
\newcommand{\ee}{\end{equation}}
\def\beq{\begin{equation}}
\def\eeq{\end{equation}}
\newcommand{\bea}{\begin{eqnarray}}
\newcommand{\eea}{\end{eqnarray}}
\newcommand{\beas}{\begin{eqnarray*}}
\newcommand{\eeas}{\end{eqnarray*}}

\newcommand{\R}{\mathbb{R}}
 
\newcommand{\N}{\mathbb{N}}

\newcommand{\cO}{\mathcal{O}}

\newcommand{\M}{\mathcal{M}}

\newcommand{\grad}{\mathrm{grad}\,}  
\newcommand{\spgrad}{\nabla_{t,x}}
\newcommand{\aform}[1]{a( #1 )}
\newcommand{\ahform}[1]{a_h( #1 )}

\allowdisplaybreaks

\def\epsilon{\varepsilon}
\def\phi {\varphi}

\def\I{\,|\,}
\DeclareMathOperator{\dis}{dist}
\def\p{\partial}

\renewcommand{\leq}{\leqslant}
\renewcommand{\geq}{\geqslant}

\makeatletter
\newcommand{\setword}[2]{%
  \phantomsection
  #1\def\@currentlabel{\unexpanded{#1}}\label{#2}%
}

\def\R{\mathbb R}

\def\N{\mathbb N}

\def\p{\partial}

\def\grad{\nabla}

\newcommand{\jump}[1]{\llbracket#1\rrbracket}
\newcommand{\tnorm}[1]{\vert\hspace{-0.3mm}\Vert#1\Vert\hspace{-0.3mm}\vert}

\title[Finite element unique continuation for the wave equation]{Space time stabilized finite element methods for a
 unique continuation problem subject to the wave equation}
\author[Burman]{Erik Burman}
\address{Department of Mathematics, University College London, Gower Street, London UK, WC1E 6BT.}
\email{e.burman@ucl.ac.uk}
\thanks{EB acknowledges funding by EPSRC grants EP/P01576X/1}
\author[Feizmohammadi]{Ali Feizmohammadi}
\address{Department of Mathematics, University College London, Gower Street, London UK, WC1E 6BT.}
\email{a.feizmohammadi@ucl.ac.uk}
\thanks{AF acknowledges funding by EPSRC grant EP/P01593X/1}
\author[M\"{u}nch]{Arnaud M\"unch}
\address{Laboratoire de Math\'ematiques Blaise Pascal, Universit\'e Clermont Auvergne, UMR CNRS 6620, Campus des C\'ezeaux,  63177~Aubi\`ere, France}
\email{arnaud.munch@uca.fr}
\author[Oksanen]{Lauri Oksanen}
\address{Department of Mathematics, University College London, Gower Street, London UK, WC1E 6BT.}
\email{l.oksanen@ucl.ac.uk}
\thanks{LO acknowledges funding by EPSRC grants EP/P01593X/1 and EP/R002207/1}

\begin{document}
\begin{abstract}
We consider a stabilized finite element method based on a spacetime
formulation, where the equations are solved on a global (unstructured)
spacetime mesh. A unique continuation problem for the wave
equation is considered, where data is known in an interior subset of spacetime. For this problem, we consider a primal-dual discrete formulation of the continuum problem with the addition of stabilization terms that are designed with the goal of minimizing the numerical errors. We  prove error estimates using the stability
properties of the numerical scheme and a continuum observability
estimate, based on the sharp geometric control condition by Bardos, Lebeau and Rauch. The order of convergence for our numerical scheme is optimal with respect to
stability properties of the continuum problem and the interpolation errors of approximating with polynomial spaces. Numerical examples are provided that illustrate the methodology.
\end{abstract}
\maketitle


\section{Introduction}
We consider a data assimilation problem for the acoustic wave equation, formulated as follows. 
Let $n \in \{1,2,3\}$, $T>0$ and $\Omega \subset \R^n$ be an open, connected, bounded set with smooth boundary $\p \Omega$. Let $u$ be the solution of the initial boundary value problem
\bel{pf}
\begin{aligned}
\begin{cases}
\Box u=\partial_t^2 u - \Delta u = 0, 
&\text{on $\M=(0,T) \times \Omega$},
\\
u =0,
&\text{on $\Sigma=(0,T) \times \p \Omega$},
\\
u|_{t=0} = u_0,\ \p_t u|_{t=0} = u_1
&\text{on $\Omega$}.
\end{cases}
    \end{aligned}
\ee
The initial data $u_0, u_1$ are assumed to be a priori unknown
functions, but 
the measurements of $u$ in some
spacetime subset $\cO=(0,T)\times \omega$, where $\omega \subset
\overline{\Omega}$ is open, is assumed to be known:
\bel{data_omega}
  \begin{aligned}
u|_{\cO}=u_\cO.
    \end{aligned}
\ee

\noindent The data assimilation problem then reads as follows:\\

\noindent \setword{(DA)}{Word:(DA)} Find $u$ given $u_\cO$.\\

The existence of a solution to the \ref{Word:(DA)} problem is always implicitly guaranteed in the sense that the measurements $u_{\cO}$ correspond to a physical solution to the wave equation \eqref{pf}. On the other hand, assuming that 
\bel{uniqueness_T}
    \begin{aligned}
T > 2 \max \{\dis(x,\omega) \I x \in \overline \Omega \},
    \end{aligned}
\ee
it follows from Holmgren's unique continuation theorem that the solution to \ref{Word:(DA)} is unique. Although uniquely solvable, \ref{Word:(DA)} might have poor stability properties if only (\ref{uniqueness_T}) is assumed. We will require the \ref{Word:(DA)} problem to be Lipschitz stable, and for this reason we make the stronger assumption that the so-called {\em geometric control condition} holds. This condition originates from \cite{BLR,BLRII} and we refer the reader to these works for the precise definition. Roughly speaking, the condition requires that all geometric optic rays in $\mathcal M$, taking into account their reflections at boundary, intersect the set $(0,T)\times \omega$.  

We recall the following formulation of the observability estimates appearing in \cite[Theorem 3.3]{BLRII} and \cite[Proposition 1.2]{LLTT17}. For the explicit derivation of this version of the estimate, we refer the reader to \cite[Theorem 2.2]{BFO18}.

\begin{thm} 
\label{continuum}
Let $\omega \subset \overline{\Omega}$, $T>0$ and suppose that $(0,T)\times \omega$
satisfies the geometric control
condition. If $u(0,\cdot) \in
L^2(\Omega)$, $\partial_t u (0,\cdot) \in H^{-1}(\Omega)$, $u|_{(0,T)
  \times \partial \Omega}\in L^2(\Sigma)$, and $\Box u \in H^{-1}(\mathcal{M})$, then $$u \in C^1([0,T];H^{-1}(\Omega))\cap C([0,T];L^2(\Omega)),$$ and
$$\sup_{t \in [0,T]}\biggl(\|u(t,\cdot)\|_{L^2(\Omega)}+\|\p_t u(t,\cdot)\|_{H^{-1}(\Omega)}\biggr)  \leq C\left( \|u\|_{L^2(\cO)} + \|\Box u\|_{H^{-1}(\mathcal{M})}+\|u\|_{L^2(\Sigma)}\right),$$
where $C>0$ is a constant depending on $\M$ and $\omega$.
\end{thm}

Let us remark that the geometric control condition is sharp in the
sense that Theorem~\ref{continuum} fails to hold if the geometric
control condition does not hold on the set $(0,T)\times \omega$
\cite{BLR}. 

The objective of the paper is to design a stabilized spacetime finite element
method for the data assimilation problem \ref{Word:(DA)}, which allows
for higher order approximation spaces. The method will also allow
for an error analysis exploiting
the stability of Theorem \ref{continuum} and the accuracy of the
spaces in an optimal way. To the best of our knowledge this is the
first complete numerical analysis of the data assimilation problem for
the wave equation, using high order spaces.

\subsection{Previous literature}

Spacetime methods for inverse
problems subject to the wave equation were introduced in \cite{CM15}
with an application to the control problem in \cite{CM15b}. In those
works however, the required $H^2$ regularity of the constraint
equation was respected on the level of approximation leading to an
approach using $C^1$-continuous approximation spaces in spacetime.
Herein we instead use an approach where the approximating space is
only $H^1$-conforming and we handle instabilities arising due to the
lack of conformity of the space through the addition of stabilization
terms. This approach to stabilization of ill-posed problems draws on
the works \cite{Bu13,Bu14b}, for the elliptic Cauchy
problem. In the context of time dependent problems, unique
continuation for the heat
equation was considered in \cite{BO18, BIO18}, with piecewise affine
finite elements in space and finite differences for the time
discretization. Finally using a similar low order approach, with
conventional, finite difference type time-discretization, the unique continuation and control problems for the wave
equation were considered in \cite{BFO18}
and \cite{BFO19} respectively. Another strategy requiring only $H^1$-regularity consists in reformulating the second order wave equation as a first order system; it is examined in \cite{MM19} for the corresponding controllability problem. 

The upshot in the present work is
that by using the spacetime framework, we may very
easily design a method that allows for high order approximation spaces
with provable, close to optimal, convergence rates, without
resorting to $C^1$-type approximation spaces.

There are several works that approach the data assimilation problem (DA), or its close variants, by solving a sequence of classical initial--boundary value problems for the wave equation. Such methods have been proposed independently in 
\cite{Stefanov2009a}, in the context of a particular application to Photoacoustic tomography, and in
\cite{Ramdani2010}, based on the so-called Luenberger observers algorithm first introduced in an ODE context in \cite{Auroux2005}. An error estimate for a discretization of a Luenberger observers based algorithm was proven in \cite{lHR},
giving logarithmic convergence rate with respect to the mesh size. 
Better convergence rates can be proven if a stability estimate is available on a scale of discrete spaces. Such discrete estimates were first derived in \cite{IZ} and we refer the reader to
the survey articles \cite{Z,EZ} and the monograph  \cite{Ervedoza2013}, as well as the recent paper \cite{EMZ} for more
details. Optimal-in-space discrete estimates can be derived from continuous estimates \cite{Miller}, however, spacetime optimal discrete estimates are known only for specific situations.  

The data assimilation problem (DA) can also be solved using the
quasi-reversibility method. This method originates from \cite{lLL}, and it has been applied to data assimilation problems subject to the wave equation in \cite{KM, KR}, and more recently to the Photoacoustic tomography problem in \cite{lCK}. We are not aware of any works proving sharp convergence rates for the quasi-reversibility method with respect to mesh size. 

The data assimilation problem (DA) arises in several applications. We mentioned already Photoacoustic tomography (PAT), and refer to \cite{Wang2009} for physical aspects of PAT, to \cite{Kuchment2008} for a mathematical review,
and to \cite{Acosta2015, lCO, Nguyen2015, Stefanov2015b} for the PAT problem in a cavity, the case closest to (DA).
Another interesting application is given in \cite{lBDP} where an obstacle detection problem is solved by using a level set method together with the quasi-reversibility method applied to a variant of (DA).

\subsection{Outline of the paper}
The paper is organized as follows. In Section~\ref{notations} we introduce a few notations used in the paper. In Section~\ref{disc_int_sec}, we start by introducing the mixed spacetime mesh followed by the discrete representation of \ref{Word:(DA)} in terms of a primal dual Lagrangian formulation. The Euler-Lagrange equations are studied, showing in particular that there exists a unique solution to the discrete formulation of \ref{Word:(DA)}. Section~\ref{error_estimate} is concerned with proving the convergence rates for the numerical error functions corresponding to the primal and dual variables. Finally, in Section~\ref{num_exp_sec} we provide two numerical examples that illustrate the theory while also making a comparison with the $H^2$-conformal finite element method introduced in \cite{CM15}.  

\section{Notations}
\label{notations}

We write 
$
\spgrad u = (\partial_t u, \nabla u),
$
where $\nabla u\in \R^n$ is the usual gradient with respect to the space
variables. The wave operator may be written as
\[
\Box u = -\spgrad \cdot (A \spgrad u),
\]
where $A$ is the matrix associated to the Minkowski metric in $\R^{1+n}$, that is, 
\[
A = \left[\begin{array}{cc} 
-1 & 0_{[1,n]} \\
0_{[n,1]} & \mathbb{I}_{[n,n]}
\end{array} \right]
\]
with $\mathbb{I}_{[n,n]}$ denoting the $n \times n$ identity matrix.
We introduce the notations
$$
(u,v)_{\mathcal M} = \int_0^T\!\!\!\int_{\Omega} u(t,x) v(t,x) \,dx\,dt, \quad \|u\|_{\mathcal M}=(u,u)^{\frac{1}{2}}_{\mathcal M}
$$
and
$$
(u,v)_{\mathcal O} = \int_0^T\!\!\!\int_{\omega} u(t,x) v(t,x) \,dx\,dt,\quad \|u\|_{\mathcal O}=(u,u)^{\frac{1}{2}}_{\mathcal O}
$$
and use an analogous notation for inner products over other subsets of $\overline{\mathcal M}$, with the understanding that the natural volume measures are used in the case of subdomains. 
We also use the shorthand notation 
\begin{align*}
\aform{u,z} = (A \spgrad u, \spgrad z)_{\mathcal{M}},
\end{align*}
and note in passing that given any $(u_0,u_1) \in H^1_0(\Omega)\times L^2(\Omega)$, the solution $u$ to equation \eqref{pf}
satisfies
\[
a(u,v)=0, \quad \forall v \in H^1_0(\mathcal{M}).
\]

Finally, we fix an integer $p$, and make the standing assumption that the continuum solution $u$ to \ref{Word:(DA)} satisfies
$$ u \in H^{p+1}(\M) \quad \text{for some $p \in \N$}.$$
The index $p$ defined above will correspond to the highest order of the spacetime polynomial approximations that can be used for the discrete solution to \ref{Word:(DA)}, while still getting optimal convergence for the numerical errors.

\section{Discrete formulation of the data assimilation problem}
\label{disc_int_sec}

Let us start this section by observing that the solution $u$ to the data assimilation problem \ref{Word:(DA)} can be obtained by analyzing the saddle points for the continuum Lagrangian functional $\mathcal J(u,z)$ that is defined through
$$ \mathcal J(u,z)=\frac{1}{2}\|u-q\|^2_{\mathcal O}+a(u,z),$$
for any 
$$ u \in H^1(0,T;L^2(\Omega))\cap L^2(0,T;H^1_0(\Omega))\quad \text{and}\quad  z \in H^1_0(\M).$$
Here, the wave equation is imposed on the primal variable $u$ by introducing a Lagrange multiplier $z$. It is easy to verify the the solution $u$ to \ref{Word:(DA)} together with $z=0$ is a saddle point for the Lagrangian functional. 

Motivated by this example, we would like to present a discrete Lagrangian functional to numerically solve \ref{Word:(DA)}. We will use discrete stabilizer (also called regularizer) terms that guarantee the existence of a unique critical point. These terms will be designed with the goal of minimizing the numerical error functions for the primal and dual variables. 

We begin by introducing the spacetime mesh in Section~\ref{spacetime_sec}. Then a discrete version of the bi-linear functional $a(\cdot,\cdot)$ is provided in Section~\ref{weak_discrete_sec}, followed by the introduction of the stabilization terms in Section~\ref{stabilizer_sec}. Finally, we present the discrete formulation of \ref{Word:(DA)} in Section~\ref{lagrangian_sec}.

\subsection{Spacetime discretization}
\label{spacetime_sec}
In this section we will introduce the spacetime finite element method
that we propose.  The method is using an $H^1$-conforming piecewise
polynomial space defined on a spacetime triangulation that can
consist of simplices, or prisms. Herein for simplicity we restrict the
discussion to the simplicial case. 
To be able to handle the case of curved boundaries without
complicating the theory with estimations of the error in the
approximation of the geometry we impose boundary conditions using a
technique introduced by Nitsche \cite{Nit71}. See also \cite[Theorem 2.1]{Thom97}
for a discussion of the application of the method to curved boundaries.

Consider a family $\mathcal T = \{\mathcal{T}_h;\ h > 0\}$ of quasi
uniform triangulations of $\mathcal{M}$ consisting of simplices $\{K\}$
such that the intersection of any two distinct simplices is
either a common vertex, a common edge or a common face. We let $h_K=
\mbox{diam}(K)$ and $h =
\max_{K \in\mathcal{T}} h_K$, (see e.g. \cite[Def. 1.140]{Ern2004}). By quasi-uniformity $h/h_K$ is uniformly
bounded, and therefore for simplicity the quasi-uniformity constant will be set to
one below. 
Observe that
we do not consider discretization of the smooth boundary $\p \mathcal M$, but instead we allow triangles adjacent to the boundary
to have curved faces, fitting $\mathcal M$. Finally, given any $k \in \N$, we let $V^k_h$ be the $H^1(\mathcal{M})$-conformal approximation space of polynomial
degree less than or equal to $k$, that is, 
\begin{align}
\label{def_Vh}
V^k_h = \{u \in C(\M):\ u\vert_K \in \mathbb{P}_k(K),\  
\forall K \in \mathcal{T}_h\},
\end{align}
where $\mathbb{P}_k(K)$ denotes the set of polynomials of degree less
than or equal to $k \ge 1$ on $K$. 

Next, we record two inequalities that will be used in the paper. The family $\mathcal{T}$
satisfies the following trace inequality, see e.g. \cite[Eq. 10.3.9]{BS08},
\begin{equation}\label{trace_cont}
\|u\|_{L^2(\partial K)} \lesssim h^{-\frac12} \|u\|_{L^2(K)} + h^{\frac12}
\|\nabla u\|_{L^2(K)}, \quad u \in H^1(K),
\end{equation}

The family $\mathcal{T}$
also satisfies the following discrete inverse inequality,
see e.g. \cite[Lem. 1.138]{Ern2004},
\begin{equation}
\label{inverse_disc}
\|\nabla u\|_{L^2(K)} \lesssim
h^{-1} \|u\|_{L^2(K)} , \quad u \in \mathbb{P}_p(K).
\end{equation}

\begin{rmk}
We will use the notations $A\lesssim B$ (resp., $A\gtrsim B$) to imply that there exists a constant $C>0$ independent of the spacetime mesh parameter $h$, such that the inequalities 
$A \leq CB$ (resp., $A\geq CB$) hold.
\end{rmk}

\begin{rmk}
It is also possible to consider the space $V^k_h \cap C^1(\M)$ with $k \geq 2$, for
the approximation of the primal variable. In this case the method
coincides with that of \cite{CM15} and the analysis shows that
optimal error estimates are satisfied also in this case.
\end{rmk}

\begin{rmk}
We will use the approximation space $V_h^p$ for the primal variable $u$. We will also fix an integer $q\leq p$ and use the space $V_h^q$ for the approximation of the dual variable $z$. As we will see, using our method, the approximation space of the dual variable can be quite coarse without sacrificing any rate of convergence for the discrete primal variable (i.e we can take $q=1$).
\end{rmk}

\subsection{A discrete bi-linear formulation for the wave equation}
\label{weak_discrete_sec}

Since no boundary conditions are imposed on the space $V^p_h$, the form
$a(u,z)$ needs to be modified on the discrete
level. For the formulation to remain consistent we propose the
following 
modified bilinear form on $V^p_h \times V^q_h$,
\[
\ahform{u_h,z_h} = \aform{u_h,z_h} - (A \spgrad u_h \cdot n_{\partial \mathcal M},
z_h)_{\partial \mathcal{M}} -  (\grad z_h \cdot n_{\partial \Omega}, u_h)_{\Sigma}.
\]
Here $n_{\partial \mathcal M}$ and $n_{\partial \Omega}$ are the outward unit normal vectors on $\p \mathcal M$ and $\p \Omega$ respectively. 
The last term in the right hand side is added to make the weak form of
the Laplace operator symmetric even in the case where no boundary
conditions are imposed on the discrete spaces. Depending on how the
stabilizing terms are chosen below, this term is not
strictly necessary in this work, but becomes essential if the
formulation must be consistent also for the adjoint equation. 

Observe that, using integration by parts, there holds 
\begin{equation}\label{eq:const_ah}
\ahform{u,z_h} = (\Box u, z_h)_{\mathcal{M}}- \underbrace{ (\grad z_h \cdot n_{\partial \Omega}, u)_{\Sigma}}_{=0} = (\Box u, z_h)_{\mathcal{M}}
\end{equation}
for all $u \in
H^2(\mathcal{M}) \cap L^2(0,T; H^1_0(\Omega))$.

\subsection{Formulation of the discrete stabilization terms for primal and dual variables}
\label{stabilizer_sec}

We denote by
$\mathcal{F}_h$ the set of internal faces of $\mathcal{T}_h$,
and define for $F \in \mathcal{F}_h$, the jump of a scalar quantity $u$
over $F$ by 
\[
\jump{ u}_F = u\vert_{K_1} - u\vert_{K_2}
\]
where $K_1, K_2 \in \mathcal{T}_h$ are the two simplices satisfying 
$K_1 \cap K_2 = F$. For vector valued quantities $u$ we define
\begin{align*}
\jump{n \cdot u}_F &= n_1 \cdot u|_{K_1}  + n_2 \cdot u|_{K_2},
\end{align*}
where $n_j$ is the outward unit normal vector of
$K_j$, $j=1,2$. Sometimes we drop the normal to alleviate the
notation. 
 The norm over all the faces in $\mathcal{F}_h$ will be
denoted by
\[
\|v\|_{\mathcal{F}_h} = \left(\sum_{F \in \mathcal{F}_h} \|v\|^2_F\right)^{\frac12},
\]
and the norm over all the simplices $\{K\}$ will be denoted by
$$
\|v\|_{\mathcal{T}_h} = \left(\sum_{K\in \mathcal{T}_h} \|v\|_K^2\right)^{\frac12}.
$$ 
For each $K \in \mathcal T_h$, we define the elementwise stabilizing form
\begin{equation}\label{eq:forward_stab}
s_K(u_h,u_h) =\|h \Box u_h \|_K^2+ \|h^{-\frac12}
u_h\|_{\partial K \cap \Sigma}^2+ \sum_{F \in \partial K \cap \mathcal{F}_h} \|
h^{\frac12} \jump{A \spgrad u_h}\|_F^2.
\end{equation}
A dual stabilizer is defined by 
\begin{equation}\label{eq:dual_stab2}
s_K^*(z_h,z_h) = \|\spgrad z_h\|_{K}^2+\|h^{-\frac12}
z_h\|^2_{\partial K \cap \partial
  \mathcal{M}}.
\end{equation}

Subsequently, the global stabilizers are defined by summing over all the elements:
\[
s= \sum_{K \in \mathcal{T}_h} s_K \quad \mbox{and} \quad s^*= \sum_{K \in \mathcal{T}_h} s^*_K
\]
and we define the semi norms $|u|_s=s(u,u)^{\frac12}$ and $|z|_{s^*} = s^*(z,z)^{\frac12}$. Observe that the following stability estimate holds:
\begin{equation}\label{lem:stab_dual}
|w_h|_{s^*} \lesssim \|\spgrad w_h\|_{\mathcal{M}} + \|h^{-\frac12}
w_h\|_{\partial \mathcal{M}}\quad \forall \, w_h \in V_h^q.
\end{equation}
We also point out for future reference that for a solution to the data
assimilation problem $u \in H^2(\M)$ and all $u_h \in V_h^p$, there holds 
    \begin{align}\label{s_on_u}
s(u-u_h,u-u_h) = s(u_h,u_h).
    \end{align}

\begin{rmk}
There is some freedom in the choices of the discrete regularization terms that yield the same error estimates as in Theorem~\ref{t1} below. The choice is more flexible for the dual variable $z$ since the continuum analogue for this variable is zero. For instance we can define the following stabilization term for the dual variable:
\begin{equation}\label{eq:dual_stab1}
s_K^*(z_h,z_h) =
s_K(z_h,z_h)+ \|h^{-\frac12} z_h\|_{\partial K \cap (\partial
  \mathcal{M}\setminus \Sigma)}+\|h^{\frac12}\partial_t z_h\|_{\partial K \cap (\partial
  \mathcal{M}\setminus \Sigma)}^2.
\end{equation}

It is also possible to use a stabilization that is exclusively carried
by the faces of the computational mesh provided $q \in \{p-2,p-1,p\}$. In this case we define
\begin{equation}\label{eq:jump_stab2}
s_K(u_h,u_h) = \sum_{F \in \partial K \cap \mathcal{F}_h} \left(\| h^{\frac12}
\jump{A \spgrad u_h}\|_F^2 + \|h^{\frac32}_F \jump{\Box u_h} \|_F^2\right)+ \|h^{-\frac12}
u_h\|_{\partial K \cap \Sigma}^2.
\end{equation}
In the second jump term we are allowed to split the operator in the
time derivative and the Laplace operator (or second order derivatives
in space) without sacrificing stability or consistency since
\[
\|h^{\frac32}_F \jump{\Box u_h} \|_F^2 \leq \|h^{\frac32}_F \jump{\partial_t^2 u_h} \|_F^2+\|h^{\frac32}_F \jump{\Delta u_h} \|_F^2.
\]
Weak consistency of the right order still holds since for a sufficiently smooth solution $u$ 
$$
 \|h^{\frac32}_F \jump{\partial_t^2 u} \|_F^2+\|h^{\frac32}_F \jump{\Delta u} \|_F^2
 = 0.
$$
\end{rmk}

\subsection{The discrete Lagrangian formulation for the data assimilation problem}
\label{lagrangian_sec}
Our finite element method is defined by the discrete Lagrangian functional
$$ \mathcal L:V_h^p \times V_h^q \to \R,$$
through
\begin{align}
\label{Lagrangian}
\mathcal{L}(u,z) = \frac12 \|u - u_\cO\|_{\mathcal{O}}^2 +
 \frac{\gamma}{2} \,s(u,u) - \frac{\gamma^*}{2}\, s^*(z,z)
+ \ahform{u,z},
\end{align}
where $\gamma,\gamma^*>0$ are fixed constants. 

The corresponding Euler-Lagrange equations read as follows. Find $(u_h,z_h) \in
V^p_h \times V^q_h$ such that for all $(v_h,w_h) \in
V^p_h \times V^q_h$ there holds
\begin{align}
\label{eq:C0_forward}
\ahform{u_h,w_h} - \gamma^* s^*(z_h,w_h) & = 0, \\ \label{eq:C0_adjoint}
(u_h,v_h)_{\mathcal{O}} + \gamma s(u_h,v) + \ahform{v_h,z_h} & = (u_\cO,v_h)_{\mathcal{O}} .
\end{align}

To simplify the notation we introduce the bilinear form
\[
\mathcal{A}_h[(u_h,z_h),(v_h,w_h)]=(u_h,v_h)_{\mathcal{O}} + \gamma s(u_h,v_h) +
\ahform{v_h,z_h} + \ahform{u_h,w_h} - \gamma^* s^*(z_h,w_h).
\]
The discrete problem \eqref{eq:C0_forward}-\eqref{eq:C0_adjoint} can then be recast as follows. Find $u_h,z_h \in
V^p_h \times V^q_h$ such that
\begin{equation}\label{eq:compact_form}
\mathcal{A}_h[(u_h,z_h),(v_h,w_h)]= (u_\cO,v_h)_{\mathcal{O}}, \quad \forall
(v_h,w_h) \in V^p_h \times V^q_h.
\end{equation}
We note that by \eqref{s_on_u} and  \eqref{eq:const_ah} there holds
\begin{equation}\label{eq:gal_ortho}
\mathcal{A}_h[(u - u_h,z_h),(v_h,w_h)] = 0.
\end{equation}

Define the residual norm
\[
\tnorm{(u,z)}^2_S = \|u\|_{\mathcal{O}}^2 + |u|_s^2+|z|^2_{s^*},
\]
and a continuity norm 
\[
\|u\|_{*}=\|\nabla_{t,x} u\|_{\M}+\|h^{\frac{1}{2}}A\nabla_{t,x}u\cdot n\|_{\partial \M} + \|h^{-\frac{1}{2}}u\|_\Sigma.
\]

For the purpose of our error analysis later, we also introduce a family of interpolants $\pi^k_h$, that are required to satisfy

\begin{assumption}$\pi^k_h: H^k(\M)\to V^k_h$ preserves Dirichlet boundary conditions and additionally satisfies
$$\text{$\|u-\pi^k_hu\|_{H^{m}(\M)}\lesssim h^{s-m}\|u\|_{H^{s}(\M)}$ for all $u \in H^s(\M)$ with $s=0,1,\ldots,k+1$ and $m=0,1,\ldots,s$.}$$
\end{assumption}

An example of such an interpolant is the Scott-Zhang interpolant \cite{SZ90}. For brevity, we will use the notations:
$$ \pi_h = \pi^1_h\quad \text{and}\quad \Pi_h=\pi^p_h.$$
We have the following lemma regarding the residual norm:
\begin{lem}\label{approx_residue}
Let $u \in H^{p+1}(\mathcal{M})$. There holds:
\[
\tnorm{(u- \Pi_h u,0)}_S \lesssim h^p \|u\|_{H^{p+1}(\mathcal{M})},
\]
\end{lem}

\begin{proof}
Note that
\[
\tnorm{(u- \Pi_h u,0)}_S = \|u- \Pi_h u\|_{\mathcal{O}} + |u- \Pi_h u|_s.
\]
For the first term we immediately see that
\[
 \|u- \Pi_h u\|_{\mathcal{O}} \leq  \|u- \Pi_h u\|_{\mathcal{M}}
 \lesssim h^{p+1} \|u\|_{H^{p+1}(\mathcal{M})}.
\]
To bound the contribution from the stabilization term, recall by definition that
\begin{multline*}|u-\Pi_h u|_s^2= \sum_{K \in \mathcal T_h}\left(\|h \Box (u-\Pi_hu)\|_K^2+ \|h^{-\frac12}
(u-\Pi_hu)\|_{\partial K \cap \Sigma}^2
+\sum_{F \in \partial K \cap \mathcal{F}_h} \|
h^{\frac12} \jump{A \spgrad (u-\Pi_hu)}\|_F^2\right).\end{multline*}
We proceed to bound the three terms on the right hand side. For the first term, we note that
$$ \sum_{K\in\mathcal T_h} \|h\Box(u-\Pi_hu)\|_K^2 \lesssim  \sum_{K\in\mathcal T_h} \|\nabla_{t,x}(u-\Pi_hu)\|_K^2 \lesssim h^{2p}\|u\|_{H^{p+1}(\M)}^2.$$
For the second term, we define $\Delta_F = \{K: \bar K \cap F \ne \emptyset
\}$, $\tilde \Delta_F = \{K: \bar K \cap \bar \Delta_F \ne \emptyset
\}$ and use \eqref{trace_cont} to write
\[
\|h^{-\frac12} (u - \Pi_h u)\|^2_{F} \lesssim h^{-1} \|u -
\Pi_h u\|^2_{\Delta_F} + \|\nabla_{t,x} (u -
\Pi_h u)\|^2_{\Delta_F} \lesssim h^{2p} |u|^2_{H^{p+1}(\tilde \Delta_F)}.
\]
By collecting the above local bounds and using the fact that $\tilde\Delta_F$ have finite overlaps, we conclude that
$$\sum_{K\in \mathcal T_h}  \|h^{-\frac12}
(u-\Pi_hu)\|_{\partial K \cap \Sigma}^2\lesssim  h^{2p} |u|^2_{H^{p+1}(\tilde \Delta_F)}. $$
For the last term, observe that using
\eqref{trace_cont} again, we have:
\[
\|h^{\frac12}_F \jump{A \spgrad (u - \Pi_h u)}\|^2_F \lesssim  \sum_{K \in
  \Delta_F} \biggl( \|\spgrad (u -
\Pi_h u)\|^2_{K} + h^2 \|D^2_{t,x} (u - \Pi_h u)\|^2_K \biggr) \lesssim h^{2p}  |u|^2_{H^{p+1}(\tilde \Delta_F)},
\]
where $D^2_{t,x}$ denotes the Hessian matrix consisting of second order derivatives in space and time variables. The claim follows by collecting the above local bounds analogously to the second term above.
\end{proof}
Next lemma is concerned with approximation properties of the continuity norm $\|\cdot\|_*$ defined earlier. The proof is analogous to the proof of the previous lemma and follows from the Definition of the interpolant $\Pi_h$ together with \eqref{trace_cont} and is therefore omitted.

\begin{lem}
\label{approx_star}
Let $u\in H^{p+1}(\M)$. There holds:
$$ \|u-\Pi_hu\|_* \lesssim h^p\|u\|_{H^{p+1}(\M)},$$
where we recall that
$$\|u\|_{*}=\|\nabla_{t,x} u\|_{\M}+\|h^{\frac{1}{2}}A\nabla_{t,x}u\cdot n\|_{\partial \M} + \|h^{-\frac{1}{2}}u\|_\Sigma.$$
\end{lem}

We end this section by proving that the solution to
\eqref{eq:compact_form} exists and is unique.
\begin{prop}
The Euler-Lagrange equation \eqref{eq:compact_form} has a unique solution $(u_h,z_h) \in V_h^p \times V_h^q$.
\end{prop}
\begin{proof}
Since equation \eqref{eq:compact_form} defines a square system of linear equations, existence is equivalent to uniqueness and
we only need to show that for $u_\cO\equiv 0$, the solution $(u_h,z_h)=(0,0)$ is
unique. Indeed, suppose that equation \eqref{eq:compact_form} with $u_\cO \equiv 0$ holds for some $(u_h,z_h) \in V_h^p\times V_h^q$.
First observe that
\[
\tnorm{u_h,z_h}_S^2 \lesssim  \mathcal{A}_h((u_h,z_h),(u_h,-z_h))=0.
\] 
This means that $|z_h|_{s^*} = 0$. Consequently, $z_h=0$
follows immediately by the Poincar\'e inequality. Next, considering $u_h$ and the stabilization \eqref{eq:forward_stab} we
see
that $u_h \in
C^1(\mathcal{M})$, $\Box u_h = 0$ a.e. in $\mathcal{M}$, $u_h|_{\mathcal O} =0$ and
$u_h\vert_{\Sigma} = 0$. Hence $u_h$ vanishes thanks to Theorem
\eqref{continuum}.
\end{proof}

\section{Error estimates}\label{error_estimate}
We will consider the derivation of error estimates in three
steps. First, we will establish the continuity of $\ahform$ with respect to
$\tnorm{\cdot}_S$ and $\|\cdot\|_*$ on the one hand (see Lemma~\ref{lem:cont1}) and then a
continuity for the exact solution with respect to $\tnorm{\cdot}_S$ and $H^1$ norms (see Lemma~\ref{lem:cont2}) on the other hand. Then,
we will prove convergence of the error in the $\tnorm{\cdot}_S$
norm. Finally, we will use these results to prove a posteriori and a
priori error estimates based on the observability estimate of Theorem
\ref{continuum}.
\begin{lem}\label{lem:cont1}
Let $v \in H^2(\mathcal{M}) + V^p_h$ and $w_h \in V^q_h$, $p \ge q \ge
1$, then there holds
\[
|\ahform{v,w_h}| \lesssim \|v\|_* |w_h|_{s^*}.
\]
\end{lem}
\begin{proof}
First, observe that by \eqref{trace_cont}, \eqref{inverse_disc} there holds
\begin{equation}
\label{boundary_disc_bound}
\|h^{\frac12}
A\spgrad w_h\cdot n\|_{\Sigma} \lesssim\|\spgrad
w_h\|_{\mathcal{M}}.
\end{equation}
Next, using the Cauchy-Schwarz inequality we write: 
\begin{multline*}
|\ahform{v,w_h}| \leq \|\spgrad v\|_{\mathcal{M}}\|\spgrad
w_h\|_{\mathcal{M}}+\|h^{\frac12} \spgrad v\cdot n\|_{\partial
  \mathcal{M}} \|h^{-\frac12} w_h\|_{\partial \mathcal{M}}
+\|h^{\frac12}
\spgrad w_h\cdot n\|_{\Sigma}  \|h^{-\frac12}
v\|_{\Sigma}.
\end{multline*}
The claim follows by combining the previous two bounds.

\end{proof}
\begin{lem}\label{lem:cont2}
Let $u \in H^{p+1}(\mathcal{M})$ be the exact solution of \eqref{pf} satisfying \eqref{data_omega} and let $(u_h,z_h)\in V^p_h\times V^q_h$ be the unique solution of the discrete Euler-Lagrange equation
\eqref{eq:compact_form}. There holds
\[
\|\Box (u-u_h)\|_{H^{-1}(\mathcal{M})} = \sup_{\substack{w \in
    H^1_0(\mathcal{M})\\
\|w\|_{H^1(\mathcal{M})}=1}}
\aform{u_h,w} \lesssim     \tnorm{(u_h,z_h)}_S.      
\]
\end{lem}
\begin{proof}
First observe that 
\[
\|\Box (u-u_h)\|_{H^{-1}(\mathcal{M})} = \sup_{\substack{w \in
    H^1_0(\mathcal{M})\\
\|w\|_{H^1(\mathcal{M})}=1}}
\aform{u-u_h,w}.
\]
Since $\Box u=0$, we have $\aform{u,w}=0$ for all $w \in
H^1_0(\mathcal{M})$ thus establishing the first equality in the claim. Using
\eqref{eq:C0_forward} we see that,
\begin{equation}\label{eq:auhw}
\aform{u_h,w} = \aform{u_h,w-\pi_hw}+\aform{u_h,\pi_hw} -  \ahform{u_h,\pi_hw} + \gamma^* s^*(z_h,\pi_hw).
\end{equation}
Using integration by parts in the first term of the right hand side we
see that
\begin{multline*}
\aform{u_h,w-\pi_hw}+\aform{u_h,\pi_hw} -  \ahform{u_h,\pi_hw} = (\Box
u_h,w-\pi_hw)_{\mathcal{T}_h} + \sum_{K \in \mathcal{T}_h} (A \spgrad u_h \cdot
n_{\partial K} , w - \pi_hw)_{\partial K}\\+ (A \spgrad u_h \cdot
n_{\partial \mathcal{M}}, \pi_hw)_{\partial
  \mathcal{M}} + (u_h,\nabla \pi_hw \cdot n)_{\Sigma} = I+II+III+IV. 
\end{multline*}
For the term $I$ we have
\[
|I| = |(\Box
u_h,w-\pi_hw)_{\mathcal{T}_h}| \leq \|h \Box u_h\|_{\mathcal{T}_h}(h^{-1}\|w-\pi_hw\|_{\mathcal T_h}) \lesssim \|h \Box u_h\|_{\mathcal{T}_h}\|w\|_{H^1(\M)}.
\]
Observe that the term $III$ is absorbed by the same quantity
with opposite sign in $II$, eliminating all terms $\pi_h w$ on the
boundary. Since also
$w\vert_{\partial \mathcal{M}} = 0$, we see that 
\[
|II+III| =  |\sum_{K \in \mathcal{T}_h} (A \spgrad u_h \cdot
n_{\partial K} , w - \pi_hw)_{\partial K \setminus \partial \mathcal{M}}|
\leq \|h^{\frac12} \jump{A\spgrad u_h}\|_{\mathcal{F}_h} \|w\|_{H^1(\M)},
\]
where we are using \eqref{trace_cont} and the definition of the interpolant $\pi_h$.
Finally, for the term $IV$ we use the Cauchy-Schwarz inequality to get the bound
\[
 |IV|=|(\nabla_{t,x} \pi_hw \cdot n, u_h)_{\Sigma}| \leq \|h^{-\frac12} u_h\|_{\Sigma}\, \|h^{\frac12} \nabla_{t,x} \pi_hw \cdot n\|_{\Sigma}\lesssim \|h^{-\frac12} u_h\|_{\Sigma}\,\|w\|_{H^1(\M)},
\]
where we used the bound \eqref{boundary_disc_bound} in the last step.

Collecting the above bounds we have that
\[
\aform{u_h,w-\pi_hw}+\aform{u_h,\pi_hw} -  \ahform{u_h,\pi_hw}
\lesssim (\|h \Box u_h\|_{\mathcal{T}_h} +\|h^{\frac12}
\jump{A\spgrad u_h}\|_{\mathcal{F}_h} + \|h^{-\frac12} u_h\|_{\Sigma}) \|w\|_{H^1(\M)}.
\]
Using the definition of $\tnorm{(u_h,0)}_S$, we may rewrite this as
\[
\aform{u_h,w-\pi_hw}+\aform{u_h,\pi_hw} -  \ahform{u_h,\pi_hw}  \lesssim
\tnorm{(u_h,0)}_S \|w\|_{H^1(\Omega)}.
\]
For the remaining term in the right hand side of \eqref{eq:auhw} we observe that by
Cauchy-Schwarz inequality
\[
s^*(z_h,\pi_hw) \leq \tnorm{(0,z_h)}_S\, |\pi_hw|_{s^*} .
\]
We can now use \eqref{lem:stab_dual} to deduce
\begin{equation*}
 |\pi_hw|_{s^*} \lesssim \|\nabla_{t,x} w_h\|_{\mathcal{M}} + \|h^{-\frac12}
 \pi_hw\|_{\partial \mathcal{M}} \lesssim \|w\|_{H^1(\M)}.
\end{equation*}
\end{proof}

We now prove convergence in the residual norm.
\begin{prop}\label{prop:res_conv}
Let $u$ be the solution to \eqref{pf}, satisfying
\eqref{data_omega}. Let $(u_h,z_h) \in V_h^p\times V_h^q$ with $p\ge q
\ge 1$ be the solution of \eqref{eq:compact_form}. Then
\[
\tnorm{(u-u_h,z_h)}_S \lesssim h^p \|u\|_{H^{p+1}(\mathcal{M})}.
\]
\end{prop}
\begin{proof}
Let $u-u_h = \underbrace{u-\Pi_h u}_{e_\Pi} + \underbrace{\Pi_h u - u_h}_{e_h} = e_{\Pi} + e_h$. 
Using the triangle inequality we write 
\[
\tnorm{(u-u_h,z_h)}_S \leq \tnorm{(e_\Pi,0)}_S+\tnorm{(e_h,z_h)}_S.
\]
Recalling Lemma \ref{approx_residue} we only need an estimate for 
$\tnorm{(e_h,z_h)}_S$.
There holds
\[
\tnorm{(e_h,z_h)}_S^2 \lesssim \mathcal{A}_h[(e_h,z_h),(e_h,-z_h)].
\]
Using Galerkin orthogonality \eqref{eq:gal_ortho} we have
\[
\mathcal{A}_h[(e_h,z_h),(e_h,-z_h)]= -\mathcal{A}_h[(e_\Pi,0),(e_h,-z_h)].
\]
By definition 
\[
\mathcal{A}_h[(e_\Pi,0),(e_h,-z_h)] = (e_\Pi,e_h)_{\mathcal{O}} + \gamma s(e_\Pi,e_h) 
- \ahform{e_\Pi,z_h}.
\]
As a consequence, applying the Cauchy-Schwarz inequality in the two
first terms in the right hand side and the continuity of Lemma
\ref{lem:cont1} in the last term we get the bound
\begin{multline*}
\mathcal{A}_h[(e_\Pi,0),(e_h,-z_h)]\lesssim (\|e_\Pi \|_{\mathcal{O}} +
|e_\Pi|_s + \|e_\Pi\|_*) ( \|e_h
\|_{\mathcal{O}}+|e_h|_s + |z_h|_{s^*})\\
\lesssim  (\|e_\Pi \|_{\mathcal{O}} +
|e_\Pi|_s + \|e_\Pi\|_*) \tnorm{(e_h,z_h)}_S.
\end{multline*}
Collecting the above bounds and applying Lemmas \ref{approx_residue}--\ref{approx_star} we
conclude that
\[
\tnorm{(e_h,z_h)}_S \lesssim  \|e_\Pi \|_{\mathcal{O}} +
|e_\Pi|_s + \|e_\Pi\|_*  \lesssim h^p \|u\|_{H^{p+1}(\mathcal M)}.
\]
\end{proof}
\begin{thm}
\label{t1}
Assume that the results of Theorem \ref{continuum} and Proposition
\ref{prop:res_conv} hold. Then we have the a posteriori error estimate
\[
\sup_{t \in [0,T]} \biggl(\|(u-u_h)(t,\cdot)\|_{L^2(\Omega)} + \|\partial_t
(u-u_h)(t,\cdot)\|_{H^{-1}(\Omega)}\biggr) \lesssim \left(\sum_{K \in
    \mathcal{T}} \eta_K^2 \right)^{\frac12}
\]
where 
\[
\eta_K^2 = \|u_h - u_\cO\|^2_{\mathcal{O} \cap K}+s_K(u_h,u_h) + s^*_K(z_h,z_h).
\]
In addition the following a priori error estimate holds for the primal variable\footnote{\, The convergence for the dual variable $z_h$ is given by Proposition~\ref{prop:res_conv}.}
\[
\sup_{t \in [0,T]} \biggl(\|(u-u_h)(t,\cdot)\|_{L^2(\Omega)} + \|\partial_t
(u-u_h)(t,\cdot)\|_{H^{-1}(\Omega)}\biggr)  \lesssim h^p \|u\|_{H^{p+1}(\mathcal M)}.
\]
\end{thm}

\begin{proof}
Taking the square of the inequality of Theorem \ref{continuum} we see
that with $e=u-u_h$
\[
\sup_{t \in [0,T]} (\|e(t,\cdot)\|_{L^2(\Omega)} + \|\partial_t
e(t,\cdot)\|_{H^{-1}(\Omega)})^2 \lesssim \|e\|_{\mathcal{O}}^2+\|\Box
e\|_{H^{-1}(\mathcal{M})}^2 + \|e\|_{\Sigma}^2.
\]
First we observe that
\[
\|e\|_{\mathcal{O}}^2 = \sum_{K \in \mathcal{T}_h} \|u_h - u_{\mathcal O}\|^2_{\mathcal{O} \cap K}
\]
and 
\[
\|e\|_{\Sigma}^2 = \sum_{K \in \mathcal{T}_h}
\|u_h\|^2_{\Sigma \cap K} \leq \sum_{K \in \mathcal{T}_h}
\|h^{-\frac12} u_h\|^2_{\Sigma \cap K} \leq |u_h|^2_{s}.
\]
Applying Lemma \ref{lem:cont2} we see that
\[
\|\Box
e\|_{H^{-1}(\mathcal{M})}^2 \lesssim \sum_{K \in
    \mathcal{T}_h} \eta_K^2
\]
which proves the first claim.

For the a priori error estimate observe that by definition and by \eqref{s_on_u} we have
\[
\left(\sum_{K \in \mathcal{T}_h} \eta_K^2\right)^{\frac12} \lesssim \|e\|_{\mathcal{O}} + \tnorm{(e,z_h)}_S
\]
and we conclude by applying the error bound of Proposition
\ref{prop:res_conv} to the right hand side. This concludes the proof.
\end{proof}
\begin{rmk}
\label{lowestq}
Observe that the preceding analysis shows that the order of the discretization space for the dual variable $z_h$, namely $q$, can be taken to be one without sacrificing any rate of convergence for the primal variable $u_h$. This is advantageous since then the system size only grows with increasing $p$.
\end{rmk}

\section{Numerical experiments}
\label{num_exp_sec}

We implement the stabilized finite element method introduced and analyzed in the previous sections. We also discuss the rate obtained according, notably, to the regularity of the initial condition to be reconstructed in the case $n=1$. We also compare the results from those obtained with the $H^2$-conformal finite element method introduced in \cite{CM15} which reads as follows: \\

\noindent Find $(u,z)\in V\times L^2(0,T; H^1_0(\Omega))$, with $V=C^1([0,T];H^{-1}(\Omega))\cap C([0,T];L^2(\Omega))$, solution of 

\begin{equation}
\left\{
\begin{aligned}
& (u,v)_{\mathcal{O}}+ \gamma\int_0^T (\Box u,\Box v)_{H^{-1}(\Omega)}dt + \int_0^T (z,\Box v)_{H_0^1(\Omega),H^{-1}(\Omega)}dt= (u_{\mathcal{O}},v)_{\mathcal{O}}, \quad \forall v\in V,\\
&  \int_0^T (w,\Box u)_{H_0^1(\Omega),H^{-1}(\Omega)}dt=0, \quad \forall w\in L^2(0,T;H^1_0(\Omega)),
\end{aligned}
\right.
\end{equation}
where $(\cdot,\cdot)_{H^1_0(\Omega),H^{-1}(\Omega)}$ denotes the dual pairing between $H^1_0(\Omega)$ and $H^{-1}(\Omega)$ so that 
$$
(z,\Box u)_{H_0^1(\Omega),H^{-1}(\Omega)}= \big(\nabla z, \nabla(-\Delta^{-1}   \Box u)\big)_{L^2(\Omega)}, \quad \forall z\in H^1_0(\Omega), u\in V.
$$
For any $\gamma\geq 0$, this well-posed mixed formulation is associated to the Lagrangian 
$$\widetilde{\mathcal{L}}:V\times L^2(0,T; H^1_0(\Omega))\to \mathbb{R}$$ defined as follows
\begin{equation}
\label{tildeL}
\widetilde{\mathcal{L}}(u,z)= \frac{1}{2}\Vert u-u_\cO\Vert^2_\mathcal{O}+ \frac{\gamma}{2}\Vert \Box u\Vert^2_{L^2(H^{-1})}-\int_0^T (z,\Box u)_{H_0^1(\Omega),H^{-1}(\Omega)}dt.
\end{equation}
At the finite dimensional level, the formulation reads: find $(u_h,z_h)\in V_h\times P_h$ solution of 
\begin{equation}
\label{FVh-HCT}
\left\{
\begin{aligned}
& (u_h,v_h)_{\mathcal{O}}+ \gamma h^2 (\Box u_h,\Box v_h)_{\mathcal{M}} + (z_h,\Box v_h)_{\mathcal{M}}= (u_\cO,v_h)_{\mathcal{O}}, \quad \forall v_h\in V_h,\\
& (w_h,\Box u_h)_{\mathcal{M}} =0, \quad \forall w_h\in P_h,
\end{aligned}
\right.
\end{equation}
where $V_h\subset V$ and $P_h\subset L^2(0,T; H^1_0(\Omega))$ for all
$h>0$. As in \cite{CM15}, we shall use a conformal approximation $V_h$
based on the $C^1$ triangular reduced HCT element (see
\cite{Bernadou_Hassan_1981}). Concerning the approximation of the
multiplier $z$, we consider $P_h=\{z\in C(\mathcal{M}):
u_{\vert K}\in \mathbb{P}^1(K), \forall K\in \mathcal{T}_h\}$. This
method does not enter the above framework, however if a dual
stabilizer as in \eqref{eq:dual_stab2} is added, the above
theory may be applied and leads to error bounds also in this case.

The experiments are performed with the FreeFem++ package developed at the University Paris 6 (see \cite{hecht2012}), very well-adapted to the spacetime formulation.  

\subsection{Example 1}
For our first example, we take simply $\Omega=(0,1)$ and first consider an observation $u_\cO$ based on the smooth initial condition $(u_0(x),u_1(x))=(\sin(3\pi x),0)$ completed with $T=2$ and  $\omega=(0.1,0.3)$, $\Omega=(0,1)$. Observe that the corresponding solution is simply the smooth function
$$
(\mathbf{Ex1}) \quad u(t,x)=\sin(3\pi x) \cos(3\pi t).
$$
In view of Theorem \ref{t1}, we expect a rate equal to $p$ for the primal variable $u$ when approximated with elements in $V_h^p$. Tables \ref{tab:ex1_zh}, \ref{tab:ex1_uh}, \ref{tab:ex1_u0h}, \ref{tab:ex1_u1h} provide the norms $$\Vert u-u_h \Vert_{L^2(\mathcal{M})}/\Vert u\Vert_{L^2(\mathcal{M})},\quad
\Vert (u-u_h)(0,\cdot) \Vert_{L^2(0,1)}/\Vert u(0,\cdot)\Vert_{L^2(0,1)},\quad \Vert (u-u_h)_t(0,\cdot) \Vert_{L^2(0,1)}$$ 
with respect to $h$ for the primal variable and the norm $\Vert z_h \Vert_{L^2(0,T; H^1_0(0,1))}$
with respect to $h$ for the dual variable. These are obtained from the formulation (\ref{FVh-HCT}) based on a conformal approximation with $\gamma=10^{-3}$ and from the formulation (\ref{eq:compact_form}) based on a non conformal approximation with $p,q\in \{1,2,3\}$ and $q\leq p$  (see \eqref{def_Vh}). For the latter, we use the dual stabilizer (\ref{eq:dual_stab2}) with 
$\gamma=10^{-3}$, $\gamma^\star=1$ and $h_F=h_K=h$. The linear system associated to the mixed formulation \eqref{eq:C0_forward}-\eqref{eq:C0_adjoint} is solved using a direct UMFPACK solver. 

Concerning the primal variable $u$, we obtain the following behavior $$\Vert u-u_h \Vert_{L^2(\mathcal{M})}/\Vert u\Vert_{L^2(\mathcal{M})}\approx \beta \times h^{\tau}$$ with 
\begin{equation}
\begin{aligned}
& \beta=e^{3.83}, &\tau=2.21 , \quad &(u_h,z_h)\in V_h^1\times V_h^1,\\
& \beta=e^{1.76}, &\tau=2.33 , \quad &(u_h,z_h)\in V_h^2\times V_h^1,\\
& \beta=e^{0.60}, &\tau=2.00, \quad &(u_h,z_h)\in V_h^3\times V_h^1,\\
& \beta=e^{2.63}, &\tau=2.57 , \quad &(u_h,z_h)\in V_h^2\times V_h^2,\\
& \beta=e^{0.19}, &\tau=1.98 , \quad &(u_h,z_h)\in V_h^3\times V_h^2,\\
& \beta=e^{0.33}, &\tau=2.01 , \quad &(u_h,z_h)\in V_h^3\times V_h^3,\\
& \beta=e^{0.32}, &\tau=2.99 , \quad &(u_h,z_h)\in HCT\times V_h^1,
\end{aligned}
\end{equation}
and we observe a rate close to $2$ for the approximation $V_h^p\times V_h^q$. The choice of $p$ mainly affects the constant $\beta$. In particular, we do not observe a rate equal to $3$ when $V_h^3$ is used; several reasons may explain this fact; i) numerical integration and approximation of $u_\mathcal{O}$ not taken into account in the analysis of Section \ref{error_estimate}; ii) bad conditioning of the square matrix associated to the mixed formulation \eqref{eq:C0_forward}-\eqref{eq:C0_adjoint}, iii) difficulty to mimic the $H^3(Q_T)$ regularity of $u$ from an approximation of $u_\mathcal{O}$. We also observe that the value of $q$ does not affect the rate in agreement with Remark \ref{lowestq}. Accordingly, the use of the space $V_h^2\times V_h^1$ seems very appropriate as it also leads to a reduced CPU time (see Table \ref{tab:cpu}).

On the other hand, the use of the $HCT$ element based on a $H^2(Q_T)$ approximation leads to a rate close to $3$.

Since $u_\mathcal{O}$ is a well-prepared solution, we check from Table \ref{tab:ex1_zh} that the approximation $z_h$ of the dual variable (which has the meaning of a Lagrange multiplier for the weak formulation of the wave equation) goes to zero with $h$ for the $L^2(0,T;H_0^1(0,1))$ norm. 

With respect to the role of $\gamma$ and $\gamma^\star$, here taken equal to $10^{-3}$ and $1$ respectively, we have observed the following phenomenon: when $p$ is strictly larger than $q$, i.e. when the primal variable is approximated in a richer space than the dual one, the value of $\gamma^{\star}$ has no influence on the quality of the result. In particular $\gamma^{\star}=0$ still leads to a well-posed discrete formulation and provides the same results compared to for instance $\gamma^{\star}=1$. Moreover, in that case, whatever be the value of $\gamma^\star$, $\gamma$ must be small but strictly positive; the choice $\gamma=0$ leads to a non invertible formulation. On the contrary, when the same finite element space is used for primal and dual variables, i.e. when $p=q$, we observe that the stabilization of the dual variable, i.e. $\gamma^\star>0$ is compulsory to achieve well-posedness. In that case, the choice $\gamma=0$ provides excellent results (except for $p=q=1$ and $h$ not small enough).  
We remark that similar qualitative and quantitative conclusions are observed with structured meshes.

\begin{table}[http]
	\centering
		\begin{tabular}{|c|ccccc|}
			\hline
			 			   mesh  & $\sharp 1$ & $\sharp 2$  &  $\sharp 3$ &   $\sharp 4$ & $\sharp 5$   
						   			\tabularnewline
			\hline
			$h$		  & $1.57\times 10^{-1}$  &      $8.22\times 10^{-2}$  	&  $4.03\times 10^{-2}$ 	&  $2.29\times 10^{-2}$  	&  $1.25\times 10^{-2}$ 	
			\tabularnewline
			$card(\mathcal{T}_h)$  &  $442$ &  $1750$  &  $7164$ &  $29182$ & $116300$ 	
			\tabularnewline
			$\sharp vertices$    & $252$ &     $936$ 		& $3703$ 				& $14832$ 	& $58631$ 
			\tabularnewline
			$card(V_h)$ - $\mathbb{P}_2$    & $945$ &     $3621$ 		& $14569$ 				& $58845$ 	& $233561$ 
			\tabularnewline
			$card(V_h)$ - $\mathbb{P}_3$    & $2080$ &     $8056$ 		& $32599$ 				& $132040$ 	& $524791$ 
			\tabularnewline
			$card(V_h)$ - $HCT$    & $1449$ &     $5493$ 		& $21975$ 				& $88509$ 	& $350823$ 
	                 \tabularnewline
															\hline
		\end{tabular}
		\vspace{0.1cm}
	\caption{Data of five triangular meshes associated to $\mathcal{M}=(0,1)\times (0,2)$.}
	\label{tab:1D_S_ex2_zh}
\end{table}

\begin{table}[http]
	\centering
		\begin{tabular}{|ccccccc|}
			\hline
			 			     $V_h^1\times V_h^1$ & $V_h^2\times V_h^1$  &  $V_h^2\times V_h^2$ &   $V_h^3\times V_h^1$ & $V_h^3\times V_h^2$ & $V_h^3\times V_h^3$   & $HCT\times V_h^1$
						   			\tabularnewline
			\hline
					   $1.74$  &      $4.78$  	&  $8.46$ 	&  $16.19$  	&  $19.27$ 	& $28.59$ & $10.76$
 			\tabularnewline
																		\hline
		\end{tabular}
		\vspace{0.1cm}
	\caption{$\mathbf{(Ex 1)}$; CPU time (in second) to solve (\ref{eq:C0_forward})-(\ref{eq:C0_adjoint})  with the mesh $\sharp 4$.}
	\label{tab:cpu}
\end{table}

\begin{table}[http!]
	\centering
		\begin{tabular}{|c|ccccc|}
			\hline
			 			   $h$  & $1.57\times 10^{-1}$  &      $8.22\times 10^{-2}$  	&  $4.03\times 10^{-2}$ 	&  $2.29\times 10^{-2}$  	&  $1.25\times 10^{-2}$  
						   			\tabularnewline
			\hline
			(\ref{FVh-HCT})	 - $HCT\times V_h^1$ & $6.41\times 10^{-5}$ &      $3.05\times 10^{-5}$ 	& $1.728\times 10^{-5}$	& $9.29\times 10^{-6}$ 		 		& $5.20\times 10^{-6}$ 	
			\tabularnewline
			(\ref{eq:compact_form}) - $V_h^1\times V_h^1$	   & $5.51\times 10^{-2}$ &     $5.65\times 10^{-2}$ 		& $2.20\times 10^{-2}$ 				& $8.15\times 10^{-3}$ 	& $2.89\times 10^{-3}$ 
			\tabularnewline
			(\ref{eq:compact_form}) - $V_h^2\times V_h^2$	   & $4.04\times 10^{-2}$ &     $1.14\times 10^{-2}$ 		& $2.67\times 10^{-3}$ 				& $7.26\times 10^{-4}$ 	& $1.98\times 10^{-4}$
			\tabularnewline
			(\ref{eq:compact_form}) - $V_h^3\times V_h^3$	   & $8.32\times 10^{-3}$ &     $1.21\times 10^{-3}$ 		& $2.30\times 10^{-4}$ 				& $7.95\times 10^{-5}$ 	& $4.71\times 10^{-5}$
			\tabularnewline
			(\ref{eq:compact_form}) - $V_h^2\times V_h^1$	   & $3.89\times 10^{-3}$ &     $8.30\times 10^{-4}$ 		& $2.12\times 10^{-4}$ 				& $6.84\times 10^{-5}$ 	& $2.28\times 10^{-5}$
			\tabularnewline
			(\ref{eq:compact_form}) - $V_h^3\times V_h^1$   & $1.83\times 10^{-3}$ &     $5.85\times 10^{-4}$ 		& $1.66\times 10^{-4}$ 				& $6.34\times 10^{-5}$ 	& $3.88\times 10^{-5}$
			\tabularnewline
			(\ref{eq:compact_form}) - $V_h^3\times V_h^2$	   & $3.13\times 10^{-3}$ &     $7.83\times 10^{-4}$ 		& $2.13\times 10^{-4}$ 				& $8.44\times 10^{-5}$ 	& $4.87\times 10^{-5}$
			\tabularnewline
												\hline
		\end{tabular}
		\vspace{0.1cm}
	\caption{$\mathbf{(Ex 1)}$; $\Vert z_h \Vert_{L^2(0,T;H^1_0(0,1))}$ w.r.t $h$.}
	\label{tab:ex1_zh}
\end{table}

\begin{table}[http]
	\centering
		\begin{tabular}{|c|ccccc|}
			\hline
			 			   $h$  & $1.57\times 10^{-1}$  &      $8.22\times 10^{-2}$  	&  $4.03\times 10^{-2}$ 	&  $2.29\times 10^{-2}$  	&  $1.25\times 10^{-2}$  
						   			\tabularnewline
			\hline
			(\ref{FVh-HCT})	 - $HCT\times V_h^1$	  & $6.46\times 10^{-3}$ &      $7.63\times 10^{-4}$ 	& $7.99\times 10^{-5}$	& $1..53\times 10^{-5}$ 		 		& $3.58\times 10^{-6}$ 	
			\tabularnewline
			(\ref{eq:compact_form}) - $V_h^1\times V_h^1$	   & $5.83\times 10^{-1}$ &     $2.26\times 10^{-2}$ 		& $4.75\times 10^{-2}$ 				& $1.07\times 10^{-2}$ 	& $2.31\times 10^{-3}$ 
			\tabularnewline
			(\ref{eq:compact_form}) - $V_h^2\times V_h^2$	   & $1.51\times 10^{-1}$ &     $1.90\times 10^{-2}$ 		& $2.86\times 10^{-3}$ 				& $7.68\times 10^{-4}$ 	& $2.17\times 10^{-4}$
			\tabularnewline
			(\ref{eq:compact_form}) - $V_h^3\times V_h^3$	   & $3.45\times 10^{-2}$ &     $8.88\times 10^{-3}$ 		& $2.16\times 10^{-3}$ 				& $7.04\times 10^{-4}$ 	& $3.81\times 10^{-4}$ 
			\tabularnewline
			(\ref{eq:compact_form}) - $V_h^2\times V_h^1$	   & $9.18\times 10^{-2}$ &     $1.48\times 10^{-2}$ 		& $2.80\times 10^{-3}$ 				& $8.01\times 10^{-4}$ 	& $2.42\times 10^{-4}$
			\tabularnewline
			(\ref{eq:compact_form}) - $V_h^3\times V_h^1$	   & $4.80\times 10^{-2}$ &     $1.17\times 10^{-2}$ 		& $2.81\times 10^{-3}$ 				& $9.93\times 10^{-4}$ 	& $5.34\times 10^{-4}$
			\tabularnewline
			(\ref{eq:compact_form}) - $V_h^3\times V_h^2$	   & $3.23\times 10^{-2}$ &     $8.83\times 10^{-3}$ 		& $2.16\times 10^{-3}$ 				& $7.05\times 10^{-4}$ 	& $3.82\times 10^{-4}$
			\tabularnewline

												\hline
		\end{tabular}
		\vspace{0.1cm}
	\caption{$\mathbf{(Ex 1)}$; $\Vert u-u_h \Vert_{L^2(\mathcal{M})}/\Vert u\Vert_{L^2(\mathcal{M})}$ w.r.t $h$.}
	\label{tab:ex1_uh}
\end{table}

\begin{table}[http]
	\centering
		\begin{tabular}{|c|ccccc|}
			\hline
			 			   $h$  & $1.57\times 10^{-1}$  &      $8.22\times 10^{-2}$  	&  $4.03\times 10^{-2}$ 	&  $2.29\times 10^{-2}$  	&  $1.25\times 10^{-2}$  
						   			\tabularnewline
			\hline
			(\ref{FVh-HCT})	- $HCT\times V_h^1$	  & $8.58\times 10^{-2}$ &      $2.06\times 10^{-2}$ 	& $5.08\times 10^{-3}$	& $1.26\times 10^{-3}$ 		 		& $3.16\times 10^{-3}$ 	
			\tabularnewline
			(\ref{eq:compact_form}) - $V_h^1\times V_h^1$	   & $3.58\times 10^{-1}$ &     $1.47\times 10^{-1}$ 		& $3.18\times 10^{-2}$ 				& $7.79\times 10^{-3}$ 	& $1.70\times 10^{-3}$ 
			\tabularnewline
			(\ref{eq:compact_form}) - $V_h^2\times V_h^2$	   & $1.05\times 10^{-1}$ &     $1.33\times 10^{-2}$ 		& $1.98\times 10^{-3}$ 				& $5.39\times 10^{-4}$ 	& $1.52\times 10^{-4}$
			\tabularnewline
			(\ref{eq:compact_form}) - $V_h^3\times V_h^3$	   & $2.41\times 10^{-2}$ &     $6.20\times 10^{-3}$ 		& $1.52\times 10^{-3}$ 				& $4.96\times 10^{-4}$  	& $2.68\times 10^{-4}$  
			\tabularnewline
			(\ref{eq:compact_form}) - $V_h^2\times V_h^1$	   & $4.23\times 10^{-2}$ &     $1.19\times 10^{-2}$ 		& $3.34\times 10^{-3}$ 				& $9.81\times 10^{-4}$ 	& $2.81\times 10^{-4}$
			\tabularnewline
			(\ref{eq:compact_form}) - $V_h^3\times V_h^1$	   & $4.39\times 10^{-2}$ &     $1.19\times 10^{-2}$ 		& $3.31\times 10^{-3}$ 				& $9.21\times 10^{-4}$ 	& $4.76\times 10^{-4}$
			\tabularnewline
			(\ref{eq:compact_form}) - $V_h^3\times V_h^2$	   & $2.33\times 10^{-2}$ &     $6.54\times 10^{-3}$ 		& $1.66\times 10^{-3}$ 				& $5.67\times 10^{-4}$ 	& $3.07\times 10^{-4}$
			\tabularnewline

												\hline
		\end{tabular}
		\vspace{0.1cm}
	\caption{$\mathbf{(Ex 1)}$; $\Vert (u-u_h)(\cdot,0) \Vert_{L^2(0,1)}/\Vert u(\cdot,0) \Vert_{L^2(0,1)}$ w.r.t. $h$.}
	\label{tab:ex1_u0h}
\end{table}

\begin{table}[http]
	\centering
		\begin{tabular}{|c|ccccc|}
			\hline
			 			   $h$  & $1.57\times 10^{-1}$  &      $8.22\times 10^{-2}$  	&  $4.03\times 10^{-2}$ 	&  $2.29\times 10^{-2}$  	&  $1.25\times 10^{-2}$  
						   			\tabularnewline
			\hline
			(\ref{FVh-HCT})	- $HCT\times V_h^1$	  & $1.75\times 10^{-1}$ &      $9.29\times 10^{-2}$ 	& $3.81\times 10^{-2}$	& $1.85\times 10^{-2}$ 		 		& $8.96\times 10^{-3}$ 	
			\tabularnewline
			(\ref{eq:compact_form}) - $V_h^1\times V_h^1$	   & $4.74\times 10^{-2}$ &     $7.70\times 10^{-2}$ 		& $3.20\times 10^{-2}$ 				& $1.21\times 10^{-2}$ 	& $5.02\times 10^{-3}$ 
			\tabularnewline
			(\ref{eq:compact_form}) - $V_h^2\times V_h^2$	   & $3.06\times 10^{-2}$ &     $8.06\times 10^{-3}$ 		& $1.99\times 10^{-3}$ 				& $4.77\times 10^{-4}$ 	& $1.14\times 10^{-4}$ 
			\tabularnewline
			(\ref{eq:compact_form}) - $V_h^3\times V_h^3$	   & $6.48\times 10^{-3}$ &     $8.26\times 10^{-4}$ 		& $1.79\times 10^{-4}$ 				& $7.03\times 10^{-5}$ 	& $4.35\times 10^{-5}$ 
			\tabularnewline
			(\ref{eq:compact_form}) - $V_h^2 \times V_h^1$	   & $3.46\times 10^{-2}$ &     $1.05\times 10^{-2}$ 		& $2.64\times 10^{-3}$ 				& $8.63\times 10^{-4}$ 	& $2.74\times 10^{-4}$
			\tabularnewline
			(\ref{eq:compact_form}) - $V_h^3 \times V_h^1$	   & $4.01\times 10^{-2}$ &     $1.29\times 10^{-2}$ 		& $3.78\times 10^{-3}$ 				& $1.47\times 10^{-3}$ 	& $8.35\times 10^{-4}$
			\tabularnewline
			(\ref{eq:compact_form}) - $V_h^3 \times V_h^2$	   & $5.48\times 10^{-3}$ &     $1.25\times 10^{-3}$ 		& $4.10\times 10^{-4}$ 				& $1.61\times 10^{-4}$ 	& $8.78\times 10^{-5}$
			\tabularnewline

												\hline
		\end{tabular}
		\vspace{0.1cm}
	\caption{$\mathbf{(Ex 1)}$; $\Vert (u-u_h)_t(\cdot,0) \Vert_{H^{-1}(0,1)}$ w.r.t. $h$.}
	\label{tab:ex1_u1h}
\end{table}

\begin{figure}[!http]
\begin{center}
\includegraphics[scale=0.65]{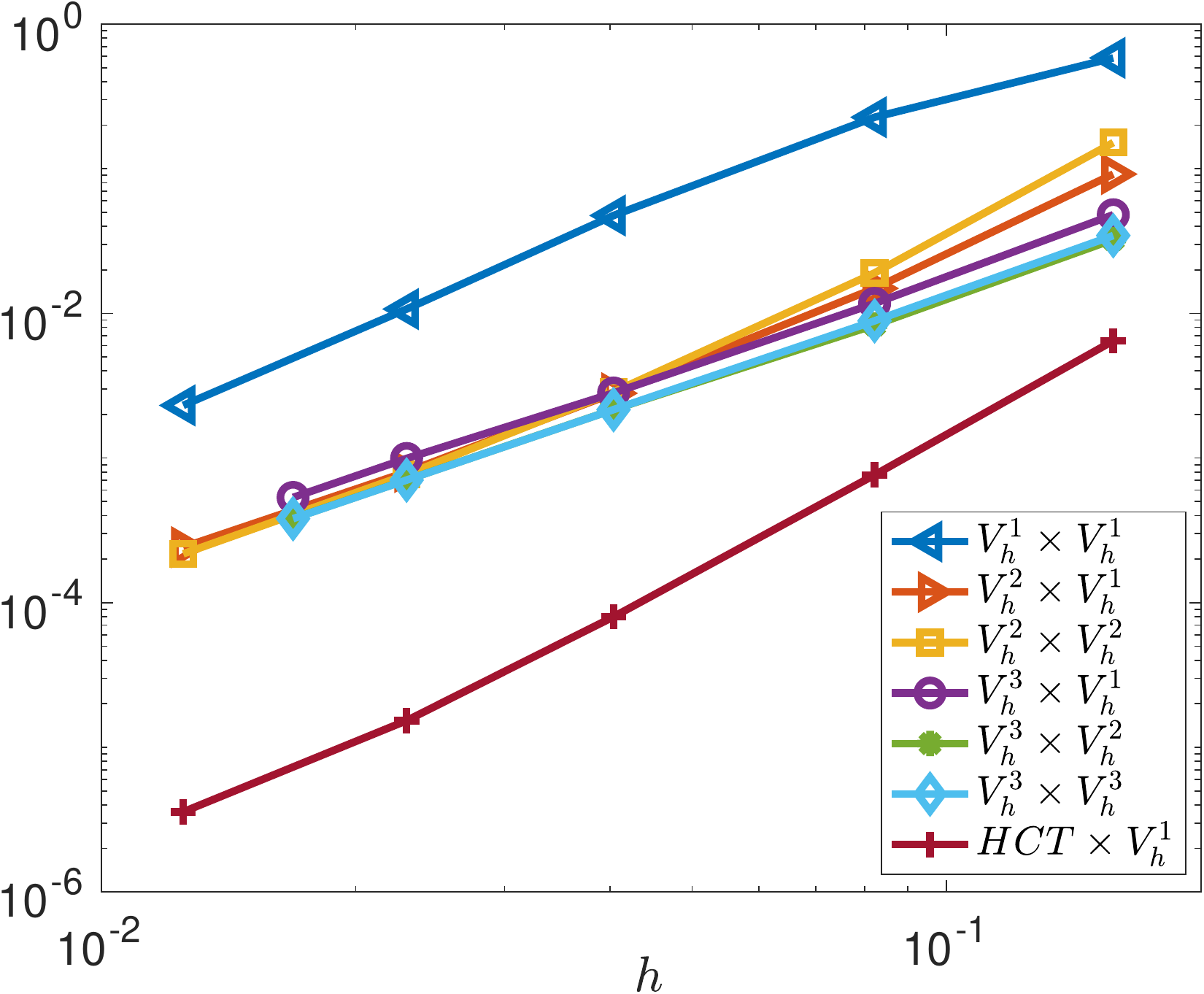}
\caption{(\textbf{Ex1})- Relative error $\Vert u-u_h \Vert_{L^2(\mathcal{M})}/\Vert u\Vert_{L^2(\mathcal{M})}$ with respect to $h$ for various approximation spaces (see Table \ref{tab:ex1_uh}).}\label{erreur_normL2_ex1-crop}
\end{center}
\end{figure}


When a richer space is used to approximate the dual variable $z$ than the primal one, we observe a locking phenomenon leading to unsatisfactory results. For instance, considering the approximation $V_h^1\times V_h^2$, and the mesh size $\sharp 4$, we obtain $\Vert z_h \Vert_{L^2(H^1_0)}=1.24\times 10^{-1}$ and $\Vert u-u_h \Vert_{L^2(\mathcal{M})}/\Vert u\Vert_{L^2(\mathcal{M})}=3.62\times 10^{-1}$ to be compare with the values $6.84\times 10^{-5}$ and $8.01\times 10^{-4}$ for the $V_h^2\times V_h^1$ approximation. This property, also observed for instance for $\gamma=1.$ and $\gamma^\star=10^{-3}$, seems independent of the values of $\gamma$ and $\gamma^{\star}$. 

To conclude this example, we emphasize that the spacetime discretization introduced in the previous sections is very well-appropriated for mesh adaptivity. Using the $V_h^2\times V_h^1$ approximation, Figure \ref{adaptmesh_ex}-left depicts the mesh obtained after seven adaptative refinements based on the local values of gradient of the primal variable $u_h$. Starting with a coarse mesh composed of $288$ triangles and $166$ vertex, the final mesh is composed with $13068$ triangles and $6700$ vertices. We obtain the following values:  $\Vert z_h \Vert_{L^2(H^1_0)}=3.76\times 10^{-3}$;  $\Vert u-u_h \Vert_{L^2(\mathcal{M})}/\Vert u\Vert_{L^2(\mathcal{M})}=3.54\times 10^{-2}$;  $\Vert (u-u_h)(0) \Vert_{L^2(0,1)}/\Vert u(0) \Vert_{L^2(0,1)}=2.48\times 10^{-2}$; $\Vert (u-u_h)_t(0) \Vert_{H^{-1}(0,1)}=3.36\times 10^{-3}$ for a CPU time equal to $2.17$.

%
%

\subsection{Example 2}


For our second numerical example, we consider the observation $u_\cO$ based on the initial condition $u_0(x)=1-\vert 2x-1\vert \in H^1_0(\Omega)$,  $u_1(x)=1_{(1/3,2/3)}(x)\in L^2(\Omega)$ and $T=2$, $\omega=(0.1,0.3)$, considered in \cite[section 5.1]{CM15}. The corresponding solution $u$ belongs to $H^1(\mathcal{M})$ but not in $H^2(\mathcal{M})$ and is given by 
$$
\mathbf{(Ex 2)} \quad
\left\{
\begin{aligned}
& u(t,x)=\sum_{k>0} \biggl(a_k \cos(k\pi t)+\frac{b_k}{k\pi}\sin(k\pi t)\biggr)\sqrt{2}\sin(k\pi t),\\ 
& a_k=\frac{4\sqrt{2}}{\pi^2 k^2} \sin(\pi k/2), \quad b_k=\frac{1}{\pi k}\big(\cos(\pi k/3)-\cos(2\pi k/3) \big), \quad k>0. 
\end{aligned}
\right.
$$

We define the observation $u_{\mathcal{O}}$ as the restriction over $(0.1,0.3)\times (0,2)$ of the first fifty terms in the previous sum. Tables \ref{tab:ex2_zh}, \ref{tab:ex2_uh}, \ref{tab:ex2_u0h}, \ref{tab:ex2_u1h} provide the norms $$\Vert u-u_h \Vert_{L^2(\mathcal{M})}/\Vert u\Vert_{L^2(\mathcal{M})},\quad
\Vert (u-u_h)(0,\cdot) \Vert_{L^2(0,1)},\quad\Vert (u-u_h)_t(0,\cdot) \Vert_{L^2(0,1)}$$ with respect to $h$ for the primal variable and $\Vert z_h \Vert_{L^2(H^1_0)}$
with respect to $h$ for the dual variable, obtained from the formulation (\ref{FVh-HCT}) and from the formulation (\ref{eq:compact_form}) with $p\in \{1,2,3\}, q\in \{1,2\}$ and $q<p$  (see \eqref{def_Vh}). We use again the dual stabilizer (\ref{eq:dual_stab2}) with 
$\gamma=10^{-3}$ and $\gamma^\star=1.$ and $h_F=h_K=h$. 

In agreement with Theorem \ref{t1}, the rate of convergence with respect to $h$ depends on the regularity of the solution $u$: concerning the primal variable $u$, we obtain the following behavior $\Vert u-u_h \Vert_{L^2(\mathcal{M})}/\Vert u\Vert_{L^2(\mathcal{M})}\approx \beta \times h^{\tau} $ with 
\begin{equation}
\begin{aligned}
& \beta=e^{-0.51}, &\tau=1.09 , \quad &(u_h,z_h)\in V_h^1\times V_h^1,\\
& \beta=e^{-0.87}, &\tau=1.29 , \quad &(u_h,z_h)\in V_h^2\times V_h^1,\\
& \beta=e^{-0.74}, &\tau=1.55, \quad &(u_h,z_h)\in V_h^3\times V_h^1,\\
& \beta=e^{-1.21}, &\tau=1.04, \quad &(u_h,z_h)\in V_h^2\times V_h^2,\\
& \beta=e^{-2.38}, &\tau=0.86, \quad &(u_h,z_h)\in HCT\times V_h^1.
\end{aligned}
\end{equation}
Since the solution to be reconstructed is only in $H^1(\mathcal{M})$, the $H^2$ approximation based on the $HCT$ composite finite element is asymptotically less accurate than the polynomial approximation  $V_h^p\times V_h^q$, $p>q$.  We also check that increasing the order of the space for the dual variable does not improve the accuracy. Moreover, we observe the same property as the first example with respect to the choice of the parameter $\lambda$ and $\lambda^\star$. 

We remark that, since the solution $u$ to be reconstructed, develops singularities along characteristic lines starting from the point $x=1/2$ (due to the initial position $u_0$) and from the points $x=1/3,2/3$ (due to the initial velocity $u_1$), the adaptative refinement of the mesh mentioned in the previous subsection is of particular interest here. Using the $V_h^2\times V_h^1$ approximation, Figure \ref{adaptmesh_ex}-left depicts the mesh obtained after ten adaptative refinements based on the local values of gradient of the primal variable $u_h$. Starting with a coarse mesh composed of $288$ triangles and $166$ vertex, the final mesh is composed with $12118$ triangles and $6213$ vertices. We obtain the following values:  $\Vert z_h \Vert_{L^2(H^1_0)}=2.36\times 10^{-5}$;  $\Vert u-u_h \Vert_{L^2(\mathcal{M})}/\Vert u\Vert_{L^2(\mathcal{M})}=1.63\times 10^{-3}$;  $\Vert (u-u_h)(0,\cdot) \Vert_{L^2(0,1)}/\Vert u(0,\cdot) \Vert_{L^2(0,1)}=9.69\times 10^{-4}$; $\Vert (u-u_h)_t(0,\cdot) \Vert_{H^{-1}(0,1)}/\Vert u_t(0,\cdot) \Vert_{H^{-1}(0,1)}=4.29\times 10^{-1}$ for a CPU time equal to $2.54$. The final mesh clearly exhibits the singularities generated by the initial data $(u_0,u_1)$. On the contrary, the refinement strategy coupled with the HCT element does not permit to capture so clearly such singularities, in particular the weaker ones starting from the point $x=1/3$ and $x=2/3$ (see \cite[Figure 1]{CM15}).


\begin{table}[http!]
	\centering
		\begin{tabular}{|c|ccccc|}
			\hline
			 			   $h$  & $1.57\times 10^{-1}$  &      $8.22\times 10^{-2}$  	&  $4.03\times 10^{-2}$ 	&  $2.29\times 10^{-2}$  	&  $1.25\times 10^{-2}$  
						   			\tabularnewline
			\hline
			(\ref{FVh-HCT})	-$HCT \times V_h^1$  & $6.21\times 10^{-5}$ &      $9.57\times 10^{-5}$ 	& $1.42\times 10^{-4}$	& $1.43\times 10^{-4}$ 		 		& $1.23\times 10^{-4}$ 	
			\tabularnewline
			(\ref{eq:compact_form}) - $V_h^1\times V_h^1$	   & $1.19\times 10^{-2}$ &     $8.21\times 10^{-3}$ 		& $3.22\times 10^{-3}$ 				& $1.63\times 10^{-3}$ 	& $8.71\times 10^{-4}$ 
			\tabularnewline
			(\ref{eq:compact_form}) - $V_h^2 \times V_h^2$	   & $4.16\times 10^{-3}$ &     $1.97\times 10^{-3}$ 		& $1.00\times 10^{-3}$ 				& $5.42\times 10^{-4}$ 	& $2.95\times 10^{-4}$
			\tabularnewline
			(\ref{eq:compact_form}) - $V_h^2 \times V_h^1$	   & $6.99\times 10^{-4}$ &     $2.85\times 10^{-4}$ 		& $1.29\times 10^{-4}$ 				& $6.95\times 10^{-5}$ 	& $3.25\times 10^{-5}$
			\tabularnewline
			(\ref{eq:compact_form}) - $V_h^3 \times V_h^1$   & $4.63\times 10^{-4}$ &     $1.60\times 10^{-4}$ 		& $5.40\times 10^{-5}$ 				& $2.31\times 10^{-5}$ 	& $1.37\times 10^{-5}$
			\tabularnewline
												\hline
		\end{tabular}
		\vspace{0.1cm}
	\caption{$\mathbf{(Ex 2)}$; $\Vert z_h \Vert_{L^2(0,T;H^1_0(0,1))}$ w.r.t. $h$.}
	\label{tab:ex2_zh}
\end{table}

\begin{table}[http!]
	\centering
		\begin{tabular}{|c|ccccc|}
			\hline
			 			   $h$  & $1.57\times 10^{-1}$  &      $8.22\times 10^{-2}$  	&  $4.03\times 10^{-2}$ 	&  $2.29\times 10^{-2}$  	&  $1.25\times 10^{-2}$  
						   			\tabularnewline
			\hline
			(\ref{FVh-HCT})	-$HCT \times V_h^1$	  & $1.97\times 10^{-2}$ &      $9.72\times 10^{-3}$ 	& $5.59\times 10^{-3}$	& $3.65\times 10^{-3}$ 		 		& $2.08\times 10^{-3}$	
			\tabularnewline
			(\ref{eq:compact_form}) - $V_h^1\times V_h^1$	   & $7.96\times 10^{-2}$ &     $3.86\times 10^{-2}$ 		& $1.79\times 10^{-2}$ 				& $9.40\times 10^{-3}$ 	& $5.01\times 10^{-3}$ 
			\tabularnewline
			(\ref{eq:compact_form}) - $V_h^2 \times V_h^2$	   & $4.16\times 10^{-2}$ &     $2.22\times 10^{-2}$ 		& $1.06\times 10^{-2}$ 				& $5.66\times 10^{-3}$ 	& $2.99\times 10^{-3}$
			\tabularnewline
			(\ref{eq:compact_form}) - $V_h^2 \times V_h^1$	   & $3..54\times 10^{-2}$ &     $1.62\times 10^{-2}$ 		& $6.94\times 10^{-3}$ 				& $3.44\times 10^{-3}$ 	& $1.24\times 10^{-3}$
			\tabularnewline
			(\ref{eq:compact_form}) - $V_h^3 \times V_h^1$   & $2.47\times 10^{-2}$ &     $9.70\times 10^{-3}$ 		& $3.83\times 10^{-3}$ 				& $1.26\times 10^{-3}$ 	& $4.87\times 10^{-4}$
			\tabularnewline
												\hline
		\end{tabular}
		\vspace{0.1cm}
	\caption{$\mathbf{(Ex 2)}$; $\Vert (u-u_h) \Vert_{L^2(\mathcal{M})}/\Vert u \Vert_{L^2(\mathcal{M})}$ w.r.t. $h$.}
	\label{tab:ex2_uh}
\end{table}

\begin{table}[http!]
	\centering
		\begin{tabular}{|c|ccccc|}
			\hline
			 			   $h$  & $1.57\times 10^{-1}$  &      $8.22\times 10^{-2}$  	&  $4.03\times 10^{-2}$ 	&  $2.29\times 10^{-2}$  	&  $1.25\times 10^{-2}$  
						   			\tabularnewline
			\hline
			(\ref{FVh-HCT})	-$HCT\times V_h^1$	  & $1.53\times 10^{-2}$ &      $8.76\times 10^{-3}$ 	& $5.06\times 10^{-3}$	& $3.24\times 10^{-3}$ 		 		& $1.90\times 10^{-3}$ 
			\tabularnewline
			(\ref{eq:compact_form}) - $V_h^1\times V_h^1$	   & $6.74\times 10^{-2}$ &     $3.33\times 10^{-2}$ 		& $1.52\times 10^{-2}$ 				& $8.23\times 10^{-3}$ 	& $4.45\times 10^{-3}$ 
			\tabularnewline
			(\ref{eq:compact_form}) - $V_h^2\times V_h^2$	   & $3.96\times 10^{-2}$ &     $2.17\times 10^{-2}$ 		& $1.03\times 10^{-2}$ 				& $5.31\times 10^{-3}$ 	& $2.64\times 10^{-3}$
			\tabularnewline
			(\ref{eq:compact_form}) - $V_h^2\times V_h^1$	   & $3.31\times 10^{-2}$ &     $1.55\times 10^{-2}$ 		& $6.45\times 10^{-3}$ 				& $2.95\times 10^{-3}$ 	& $1.58\times 10^{-3}$
			\tabularnewline
			(\ref{eq:compact_form}) - $V_h^3\times V_h^1$   & $2.18\times 10^{-2}$ &     $8.95\times 10^{-3}$ 		& $3.29\times 10^{-3}$ 				& $1.58\times 10^{-3}$ 	& $1.67\times 10^{-3}$
			\tabularnewline
												\hline
		\end{tabular}
		\vspace{0.1cm}
	\caption{$\mathbf{(Ex 2)}$; $\Vert (u-u_h)(\cdot,0) \Vert_{L^2(0,1)}/\Vert u(\cdot,0) \Vert_{L^2(0,1)}$ w.r.t. $h$.}
	\label{tab:ex2_u0h}
\end{table}

\begin{table}[http!]
	\centering
		\begin{tabular}{|c|ccccc|}
			\hline
			 			   $h$  & $1.57\times 10^{-1}$  &      $8.22\times 10^{-2}$  	&  $4.03\times 10^{-2}$ 	&  $2.29\times 10^{-2}$  	&  $1.25\times 10^{-2}$ 
						   			\tabularnewline
			\hline
			(\ref{FVh-HCT})	-$HCT \times V_h^1$	  & $9.33\times 10^{-1}$ &      $6.33\times 10^{-1}$ 	& $5.44\times 10^{-1}$	& $4.79\times 10^{-1}$ 		 		& $4.54\times 10^{-1}$ 
			\tabularnewline
			(\ref{eq:compact_form}) - $V_h^1\times V_h^1$	   & $8.99\times 10^{-1}$ &     $6.07\times 10^{-1}$ 		& $4.98\times 10^{-1}$ 				& $4.62\times 10^{-1}$ 	& $4.41\times 10^{-1}$ 
			\tabularnewline
			(\ref{eq:compact_form}) - $V_h^2\times V_h^2$	   & $4.74\times 10^{-1}$ &     $4.20\times 10^{-1}$ 		& $4.41\times 10^{-1}$ 				& $4.26\times 10^{-1}$ 	& $4.34\times 10^{-1}$
			\tabularnewline
			(\ref{eq:compact_form}) - $V_h^2\times V_h^1$	   & $4.62\times 10^{-1}$ &     $4.22\times 10^{-1}$ 		& $4.41\times 10^{-1}$ 				& $4.26\times 10^{-1}$ 	& $4.34\times 10^{-1}$
			\tabularnewline
			(\ref{eq:compact_form}) - $V_h^3\times V_h^1$   & $5.02\times 10^{-1}$ &     $4.24\times 10^{-1}$ 		& $4.44\times 10^{-1}$ 				& $4.28\times 10^{-1}$ 	& $4.27\times 10^{-1}$
			\tabularnewline
												\hline
		\end{tabular}
		\vspace{0.1cm}
	\caption{$\mathbf{(Ex 2)}$; $\Vert (u-u_h)_t(\cdot,0) \Vert_{H^{-1}(0,1)}/\Vert (u_t(\cdot,0) \Vert_{H^{-1}(0,1)}$ w.r.t. $h$.}
	\label{tab:ex2_u1h}
\end{table}

\begin{figure}[!http]
\begin{center}
\includegraphics[scale=0.65]{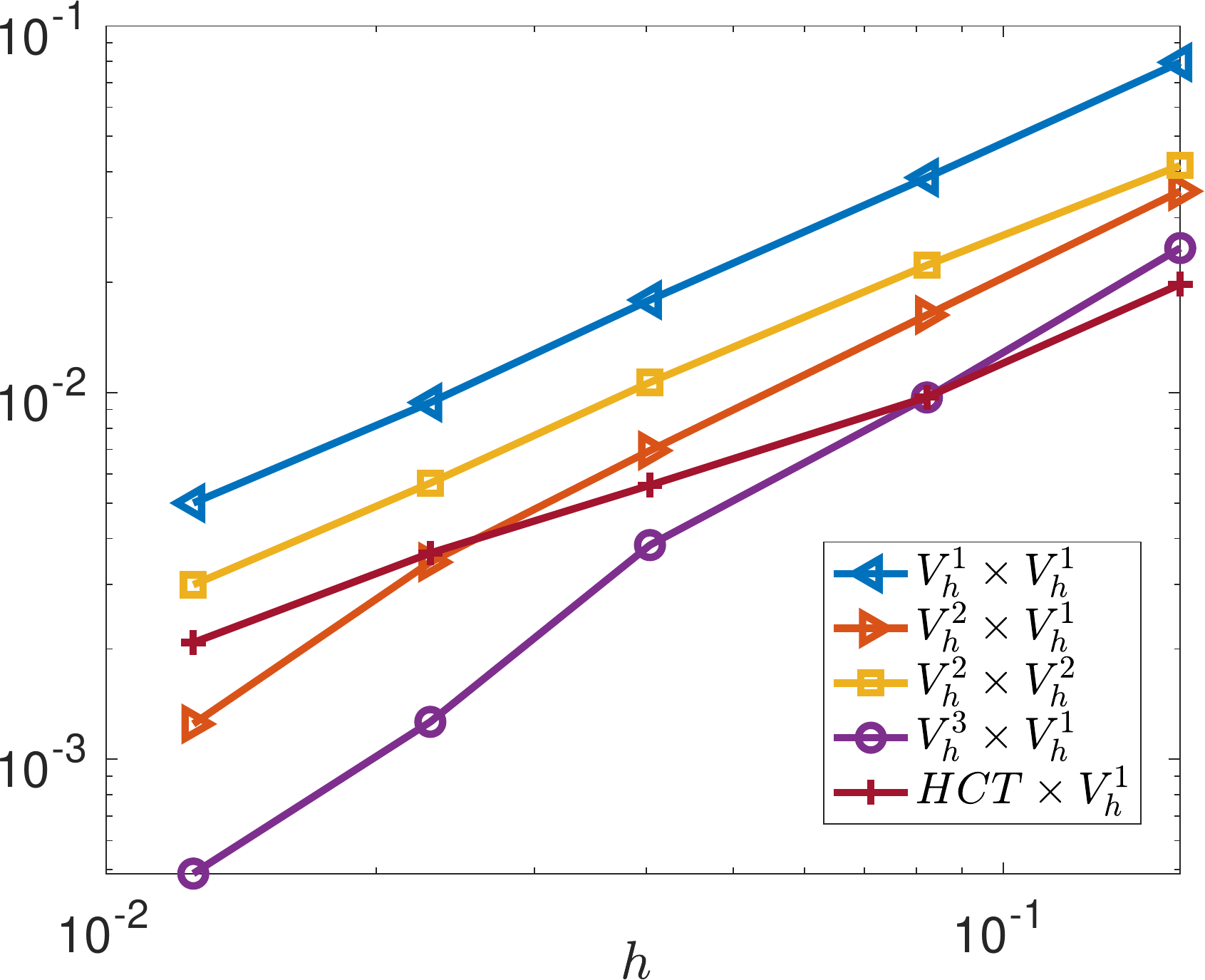}
\caption{(\textbf{Ex2})- Relative error $\Vert u-u_h \Vert_{L^2(\mathcal{M})}/\Vert u\Vert_{L^2(\mathcal{M})}$ with respect to $h$ for various approximation spaces (see Table \ref{tab:ex2_uh}).}\label{erreur_normL2_ex2-crop}
\end{center}
\end{figure}



\begin{figure}[!http]
\begin{center}
\includegraphics[scale=1.]{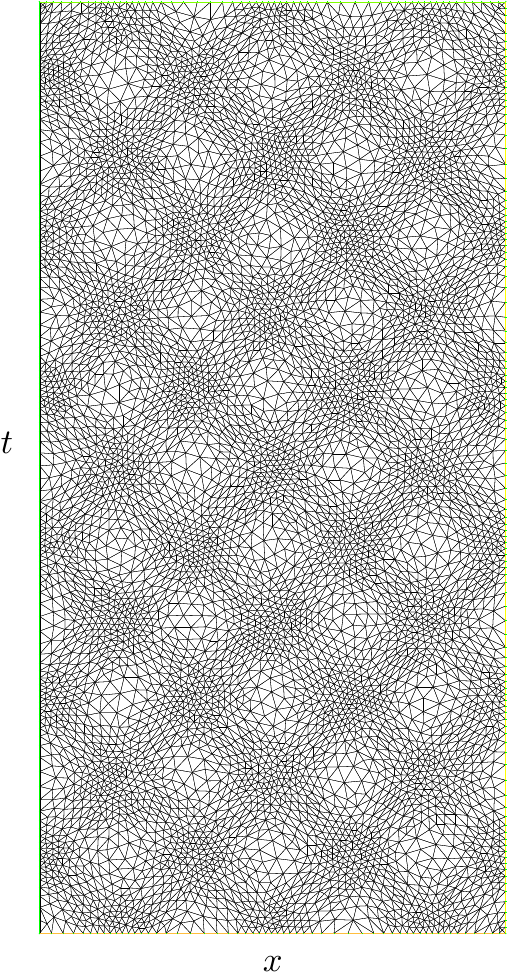}\hspace*{1.5cm}
\includegraphics[scale=0.7]{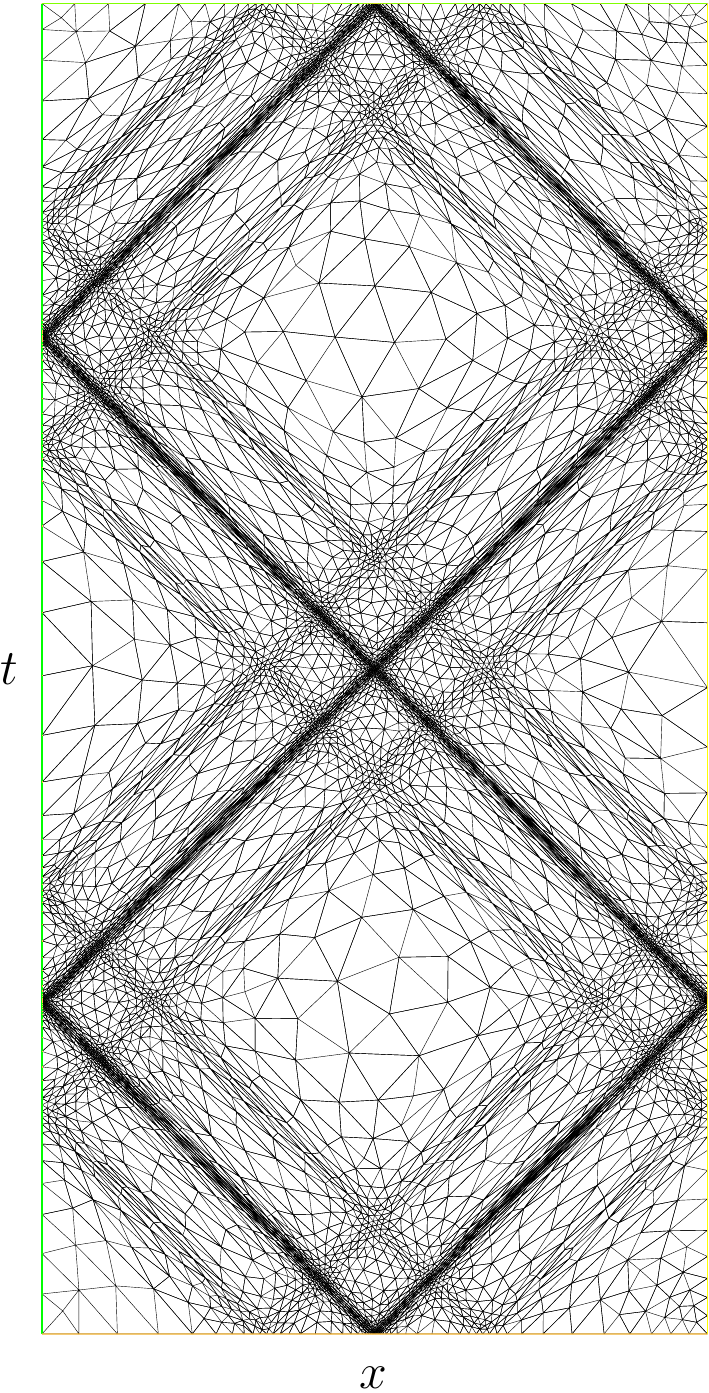}
\caption{Locally refine spacetime meshes for the example 1 (Left) and the example 2 (Right).}\label{adaptmesh_ex}
\end{center}
\end{figure}



\section{Concluding remarks}\label{conclusion}

We have introduced and analyzed a spacetime finite element
approximation of a data assimilation problem for the wave
equation. Based on an $H^1$-approximation that is nonconformal in $H^2$, the analysis yields error estimates for the natural norm $$L^\infty(0,T;L^2(\Omega))\cap H^1(0,T;H^{-1}(\Omega))$$ of order $h^p$, where $p$ is the degree of the polynomials used to describe the primal variable $u$ to be reconstructed. The numerical experiments performed for two initial data, the first one in $H^k\times H^{k-1}$ for all $k\in \mathbb{N}$, the second one in $H^1\times L^2$, exhibit the efficiency of the method. 

We emphasize that spacetime formulations are easier to implement than time-marching methods, since in particular, there is no kind of CFL condition between the time and space discretization parameters. Moreover, as shown in the numerical section, it is well-suited for mesh adaptivity. 

In comparison with the formulation introduced in \cite{CM15}, the
$H^1$-formulation of the present work does not require the
introduction of sophisticated finite element
spaces. On the other hand, the formulation requires additional
stabilized terms which are function of the jump of the gradient across
the boundary of each element, see the definition of $s_k$ in
(\ref{eq:forward_stab}). These kinds of terms are known from
non-conforming approximation of fourth order problems \cite{EGHL02}. 
So the approach can be interpreted as a non-conforming, stabilized version of the
method in \cite{CM15}.
The implementation of the stabilized terms
is not straightforward in particular, in higher dimension, and is usually not available in finite element softwares.
In such cases one can apply the so-called
orthogonal subscale stabilization \cite{RC00}, which can be shown to be
equivalent, but requires the introduction of additional degrees of
freedom, one for each component of the spacetime gradient. Another possible way to circumvent the introduction of the
gradient jump terms is to consider non-conforming approximation of the
Crouzeix-Raviart type as in \cite{Bu17nc}. A penalty is then needed on the solution jump instead to
control the $H^1$-conformity error. For the time discretization one could also explore the
possibility of using discontinuous Petrov-Galerkin methods (see for instance \cite{SZ18, EW19}).

The analysis performed can easily be extended to more general wave equations of the form $$u_{tt}-div(a(x)\nabla u) +p(x,t)u=0$$ with $a\in L^\infty(\Omega,\mathbb{R}^+_\star)$ and $p\in L^{\infty}(\mathcal{M})$ allowing to consider, through appropriate linearization techniques, data assimilation problem for nonlinear wave equation of the form $u_{tt}-\Delta u +g(u)=0$. From application viewpoint, it is also interesting to check if a spacetime approach based on a non conformal $H^1$-approximation can be efficient to address data assimilation problem from boundary observation. We refer to \cite{CM15b} where a $H^2$ conformal approximation -similar to \eqref{FVh-HCT} - is discussed, assuming that the normal derivative $\partial_{\nu}u \in L^2(\Sigma)$ is available on a part large enough of $\Sigma=(0,T)\times \partial\Omega$. Eventually, the issue of the approximation of controllability problems for the wave equation through non conformal spacetime approach can very likely be addressed as well: remember that the control of minimal $L^2(\mathcal{O}\times (0,T))$ norm for the initial data $(z_0,z_1)\in H^1_0(\Omega)\times L^2(\Omega)$ is given by $v=u 1_{\mathcal{O}}$ where $u$ together with $z$ is the saddle point of the following Lagrangian $\widehat{\mathcal{L}}: V\times L^2(0,T; H^1_0(\Omega))\to \mathbb{R}$
$$
\begin{aligned}
\widehat{\mathcal{L}}(u,z)= \frac{1}{2}\Vert u\Vert^2_\mathcal{O}&+ \frac{\gamma}{2}\Vert \Box u\Vert^2_{L^2(0,T;H^{-1}(\Omega))}-\int_0^T (z,\Box u)_{H_0^1(\Omega),H^{-1}(\Omega)} dt\\
&+ \langle u_t(\cdot,0),z_0\rangle_{H^{-1}(\Omega),H_0^1(\Omega)}-\langle u(\cdot,0),z_1\rangle_{L^2(\Omega)},
\end{aligned}
$$
very close to (\ref{tildeL}).  We refer to \cite{CCM14}.

\bibliographystyle{alpha}
\bibliography{BiblioE}

\newcommand{\etalchar}[1]{$^{#1}$}
\begin{thebibliography}{LRLTT17}

\bibitem[AB05]{Auroux2005}
Didier Auroux and Jacques Blum.
\newblock Back and forth nudging algorithm for data assimilation problems.
\newblock {\em C. R. Math. Acad. Sci. Paris}, 340(12):873--878, 2005.

\bibitem[AM15]{Acosta2015}
Sebasti\'{a}n Acosta and Carlos Montalto.
\newblock Multiwave imaging in an enclosure with variable wave speed.
\newblock {\em Inverse Problems}, 31(6):065009, 12, 2015.

\bibitem[BFO18]{BFO18}
Erik {Burman}, Ali {Feizmohammadi}, and Lauri {Oksanen}.
\newblock {A finite element data assimilation method for the wave equation}.
\newblock {\em arXiv e-prints}, page arXiv:1811.01580, Nov 2018.

\bibitem[BFO19]{BFO19}
Erik {Burman}, Ali {Feizmohammadi}, and Lauri {Oksanen}.
\newblock {A fully discrete numerical control method for the wave equation}.
\newblock {\em arXiv e-prints}, page arXiv:1903.02320, Mar 2019.
\newblock to appear Math. Comp.

\bibitem[BH81]{Bernadou_Hassan_1981}
Michel Bernadou and Kamal Hassan.
\newblock Basis functions for general {H}sieh-{C}lough-{T}ocher triangles,
  complete or reduced.
\newblock {\em Internat. J. Numer. Methods Engrg.}, 17(5):784--789, 1981.

\bibitem[BIHO18]{BIO18}
E.~Burman, J.~Ish-Horowicz, and L.~Oksanen.
\newblock Fully discrete finite element data assimilation method for the heat
  equation.
\newblock {\em ESAIM Math. Model. Numer. Anal.}, 52(5):2065--2082, 2018.

\bibitem[BLR88]{BLR}
C.~Bardos, G.~Lebeau, and J.~Rauch.
\newblock Un exemple d'utilisation des notions de propagation pour le
  contr\^{o}le et la stabilisation de probl\`emes hyperboliques.
\newblock {\em Rend. Sem. Mat. Univ. Politec. Torino}, (Special Issue):11--31
  (1989), 1988.
\newblock Nonlinear hyperbolic equations in applied sciences.

\bibitem[BLR92]{BLRII}
Claude Bardos, Gilles Lebeau, and Jeffrey Rauch.
\newblock Sharp sufficient conditions for the observation, control, and
  stabilization of waves from the boundary.
\newblock {\em SIAM J. Control Optim.}, 30(5):1024--1065, 1992.

\bibitem[BO18]{BO18}
E.~Burman and L.~Oksanen.
\newblock Data assimilation for the heat equation using stabilized finite
  element methods.
\newblock {\em Numer. Math.}, 139(3):505--528, 2018.

\bibitem[BPD19]{lBDP}
Laurent Bourgeois, Dmitry Ponomarev, and J\'{e}r\'{e}mi Dard\'{e}.
\newblock An inverse obstacle problem for the wave equation in a finite time
  domain.
\newblock {\em Inverse Probl. Imaging}, 13(2):377--400, 2019.

\bibitem[BS08]{BS08}
S.C. Brenner and L.R. Scott.
\newblock {\em The mathematical theory of finite element methods}.
\newblock Springer-Verlag, third edition, 2008.

\bibitem[Bur13]{Bu13}
Erik Burman.
\newblock Stabilized finite element methods for nonsymmetric, noncoercive, and
  ill-posed problems. {P}art {I}: {E}lliptic equations.
\newblock {\em SIAM J. Sci. Comput.}, 35(6):A2752--A2780, 2013.

\bibitem[Bur14]{Bu14b}
Erik Burman.
\newblock Error estimates for stabilized finite element methods applied to
  ill-posed problems.
\newblock {\em C. R. Math. Acad. Sci. Paris}, 352(7-8):655--659, 2014.

\bibitem[Bur17]{Bu17nc}
Erik Burman.
\newblock A stabilized nonconforming finite element method for the elliptic
  cauchy problem.
\newblock {\em Math. Comp.}, pages 75--96, 2017.

\bibitem[CCM14]{CCM14}
Carlos Castro, Nicolae C\^{\i}ndea, and Arnaud M\"{u}nch.
\newblock Controllability of the linear one-dimensional wave equation with
  inner moving forces.
\newblock {\em SIAM J. Control Optim.}, 52(6):4027--4056, 2014.

\bibitem[CK08]{lCK}
Christian Clason and Michael~V. Klibanov.
\newblock The quasi-reversibility method for thermoacoustic tomography in a
  heterogeneous medium.
\newblock {\em SIAM J. Sci. Comput.}, 30(1):1--23, 2007/08.

\bibitem[CM15a]{CM15}
Nicolae C\^{i}ndea and Arnaud M\"{u}nch.
\newblock Inverse problems for linear hyperbolic equations using mixed
  formulations.
\newblock {\em Inverse Problems}, 31(7):075001, 38, 2015.

\bibitem[CM15b]{CM15b}
Nicolae C\^{i}ndea and Arnaud M\"{u}nch.
\newblock A mixed formulation for the direct approximation of the control of
  minimal {$L^2$}-norm for linear type wave equations.
\newblock {\em Calcolo}, 52(3):245--288, 2015.

\bibitem[CO16]{lCO}
Olga Chervova and Lauri Oksanen.
\newblock Time reversal method with stabilizing boundary conditions for
  photoacoustic tomography.
\newblock {\em Inverse Problems}, 32(12):125004, 16, 2016.

\bibitem[Cod00]{RC00}
Ramon Codina.
\newblock Stabilization of incompressibility and convection through orthogonal
  sub-scales in finite element methods.
\newblock {\em Computer Methods in Applied Mechanics and Engineering},
  190(13):1579 -- 1599, 2000.

\bibitem[EG04]{Ern2004}
Alexandre Ern and Jean-Luc Guermond.
\newblock {\em Theory and practice of finite elements}, volume 159 of {\em
  Applied Mathematical Sciences}.
\newblock Springer-Verlag, New York, 2004.

\bibitem[EGH{\etalchar{+}}02]{EGHL02}
G.~Engel, K.~Garikipati, T.J.R. Hughes, M.G. Larson, L.~Mazzei, and R.L.
  Taylor.
\newblock Continuous/discontinuous finite element approximations of
  fourth-order elliptic problems in structural and continuum mechanics with
  applications to thin beams and plates, and strain gradient elasticity.
\newblock {\em Computer Methods in Applied Mechanics and Engineering},
  191(34):3669 -- 3750, 2002.

\bibitem[EMZ16]{EMZ}
Sylvain Ervedoza, Aurora Marica, and Enrique Zuazua.
\newblock Numerical meshes ensuring uniform observability of one-dimensional
  waves: construction and analysis.
\newblock {\em IMA J. Numer. Anal.}, 36(2):503--542, 2016.

\bibitem[EW19]{EW19}
Johannes Ernesti and Christian Wieners.
\newblock Space-time discontinuous {P}etrov-{G}alerkin methods for linear wave
  equations in heterogeneous media.
\newblock {\em Comput. Methods Appl. Math.}, 19(3):465--481, 2019.

\bibitem[EZ12]{EZ}
Sylvain Ervedoza and Enrique Zuazua.
\newblock The wave equation: control and numerics.
\newblock In {\em Control of partial differential equations}, volume 2048 of
  {\em Lecture Notes in Math.}, pages 245--339. Springer, Heidelberg, 2012.

\bibitem[EZ13]{Ervedoza2013}
Sylvain Ervedoza and Enrique Zuazua.
\newblock {\em Numerical approximation of exact controls for waves}.
\newblock SpringerBriefs in Mathematics. Springer, New York, 2013.

\bibitem[Hec12]{hecht2012}
F.~Hecht.
\newblock New development in {F}reefem++.
\newblock {\em J. Numer. Math.}, 20(3-4):251--265, 2012.

\bibitem[HR12]{lHR}
Ghislain Haine and Karim Ramdani.
\newblock Reconstructing initial data using observers: error analysis of the
  semi-discrete and fully discrete approximations.
\newblock {\em Numer. Math.}, 120(2):307--343, 2012.

\bibitem[IZ99]{IZ}
Juan~Antonio Infante and Enrique Zuazua.
\newblock Boundary observability for the space semi-discretizations of the
  {$1$}-{D} wave equation.
\newblock {\em M2AN Math. Model. Numer. Anal.}, 33(2):407--438, 1999.

\bibitem[KK08]{Kuchment2008}
Peter Kuchment and Leonid Kunyansky.
\newblock Mathematics of thermoacoustic tomography.
\newblock {\em European J. Appl. Math.}, 19(2):191--224, 2008.

\bibitem[KM91]{KM}
Michael~V. Klibanov and Joseph Malinsky.
\newblock Newton-{K}antorovich method for three-dimensional potential inverse
  scattering problem and stability of the hyperbolic {C}auchy problem with
  time-dependent data.
\newblock {\em Inverse Problems}, 7(4):577--596, 1991.

\bibitem[KR92]{KR}
Michael Klibanov and Rakesh.
\newblock Numerical solution of a time-like {C}auchy problem for the wave
  equation.
\newblock {\em Math. Methods Appl. Sci.}, 15(8):559--570, 1992.

\bibitem[LL67]{lLL}
R.~Latt\`es and J.-L. Lions.
\newblock {\em M\'{e}thode de quasi-r\'{e}versibilit\'{e} et applications}.
\newblock Travaux et Recherches Math\'{e}matiques, No. 15. Dunod, Paris, 1967.

\bibitem[LRLTT17]{LLTT17}
J\'{e}r\^{o}me Le~Rousseau, Gilles Lebeau, Peppino Terpolilli, and Emmanuel
  Tr\'{e}lat.
\newblock Geometric control condition for the wave equation with a
  time-dependent observation domain.
\newblock {\em Anal. PDE}, 10(4):983--1015, 2017.

\bibitem[Mil12]{Miller}
Luc Miller.
\newblock Resolvent conditions for the control of unitary groups and their
  approximations.
\newblock {\em J. Spectr. Theory}, 2(1):1--55, 2012.

\bibitem[MM19]{MM19}
Santiago Montaner and Arnaud M\"{u}nch.
\newblock Approximation of controls for linear wave equations: A first order
  mixed formulation.
\newblock {\em Mathematical control and related fields}, 9(4):729--758, 2019.

\bibitem[Nit71]{Nit71}
J.~Nitsche.
\newblock \"{U}ber ein {V}ariationsprinzip zur {L}\"osung von
  {D}irichlet-{P}roblemen bei {V}erwendung von {T}eilr\"aumen, die keinen
  {R}andbedingungen unterworfen sind.
\newblock {\em Abh. Math. Sem. Univ. Hamburg}, 36:9--15, 1971.
\newblock Collection of articles dedicated to Lothar Collatz on his sixtieth
  birthday.

\bibitem[NK16]{Nguyen2015}
Linh~V. Nguyen and Leonid~A. Kunyansky.
\newblock A dissipative time reversal technique for photoacoustic tomography in
  a cavity.
\newblock {\em SIAM J. Imaging Sci.}, 9(2):748--769, 2016.

\bibitem[RTW10]{Ramdani2010}
Karim Ramdani, Marius Tucsnak, and George Weiss.
\newblock Recovering and initial state of an infinite-dimensional system using
  observers.
\newblock {\em Automatica J. IFAC}, 46(10):1616--1625, 2010.

\bibitem[SU09]{Stefanov2009a}
Plamen Stefanov and Gunther Uhlmann.
\newblock Thermoacoustic tomography with variable sound speed.
\newblock {\em Inverse Problems}, 25(7):075011, 16, 2009.

\bibitem[SY15]{Stefanov2015b}
Plamen Stefanov and Yang Yang.
\newblock Multiwave tomography in a closed domain: averaged sharp time
  reversal.
\newblock {\em Inverse Problems}, 31(6):065007, 23, 2015.

\bibitem[SZ90]{SZ90}
L.~R. Scott and S.~Zhang.
\newblock Finite element interpolation of nonsmooth functions satisfying
  boundary conditions.
\newblock {\em Math. Comp.}, 54(190):483--493, 1990.

\bibitem[SZ18]{SZ18}
Olaf Steinnach and Marco Zank.
\newblock A stabilized space-time finite element method for the wave equation.
\newblock {\em Technische Universit\"at Graz Report 2018/5}, pages 1--27, 2018.

\bibitem[Tho97]{Thom97}
V.~Thom{\'e}e.
\newblock {\em {G}alerkin finite element methods for parabolic problems},
  volume~25 of {\em Springer Series in Computational Mathematics}.
\newblock Springer-Verlag, Berlin, 1997.

\bibitem[Wan09]{Wang2009}
Lihong~V Wang.
\newblock {\em Photoacoustic imaging and spectroscopy}.
\newblock CRC press, 2009.

\bibitem[Zua05]{Z}
Enrique Zuazua.
\newblock Propagation, observation, and control of waves approximated by finite
  difference methods.
\newblock {\em SIAM Rev.}, 47(2):197--243, 2005.

\end{thebibliography}


\ifdraft{
\listoftodos
}{}

\end{document}